\documentclass[10pt]{article}

\usepackage{latexsym,color,amsmath,amsthm,amssymb,amscd,amsfonts}
\usepackage{color}
\usepackage{dsfont}
\usepackage{yfonts}
\usepackage{pst-node}
\usepackage{tikz-cd}

\newtheorem{theorem}{Theorem}
\newtheorem{corollary}{Corollary}

\newtheorem{lemma}{Lemma}

\newtheorem{proposition}{Proposition}

\newtheorem{fact}{Fact}

\def\R{{\mathbb R}}

\def\inte{{\rm int}}
\def\eps{{\varepsilon}}

\def\conv{{\rm conv}}
\def\vol{{\rm vol}}

\def\1{\mathds{1}}

\def\supp{{\rm supp}}
\def\be{\begin{equation}}
\def\ee{\end{equation}}

\def\Lc{\mathcal{L}}

\begin{document}

	\title{The L\"{o}wner function of  a log-concave function 
		}
	
	\author{Ben Li {\thanks {Partially supported by NSF grants DMS-1600124 and DMS-1700168, and by ERC grant-770127}} , Carsten Sch\"utt  and Elisabeth M. Werner
		\thanks{Partially supported by  NSF grant DMS-1811146 }}
	
	\date{}

	\maketitle
	
	\begin{abstract}
	We  introduce the notion of L\"{o}wner (ellipsoid) function for a log-concave function and  show that it is an
	extension of the L\"{o}wner ellipsoid for convex bodies. We investigate its  duality relation to the recently defined 
	John (ellipsoid) function \cite{Alonso-Gutiérrez2017}. For convex bodies, John and L\"{o}wner ellipsoids are dual to each other.
	Interestingly, this need not be the case for the John  function and the L\"{o}wner  function.

	\end{abstract}
	\vskip 3mm
\section{ Introduction}

Asymptotic convex geometry studies the 
properties of  convex bodies
with  emphasis on the dependence  of geometric and analytic invariants on the dimension.
The convexity assumption enforces concentration of volume in a canonical way and
it is a main question if under natural normalizations  the answers to fundamental questions are independent of the dimension.
\par
\noindent
The most classical normalizations  of convex
bodies arise as solutions of extremal problems. These normalizations include the  isotropic position, which arose from classical mechanics of the 19th
century and 
which is related to a famous open problem in 
convex geometry, the hyperplane conjecture (see, e.g., the survey \cite{KlartagWerner}).  
The best results currently available there are due Bourgain \cite{Bourgain1991} and Klartag \cite{Klartag06}.
 \newline
Other positions   are the John position, also called maximal volume ellipsoid position 
and  the L\"owner position, also called minimal volume ellipsoid position.
The right choice of a position is important for
the study of affinely invariant quantities
and their related isoperimetric inequalities.
For instance, John and L\"owner position are related to the 
Brascamp-Lieb inequality and its reverse \cite{Ball1989, Fbarthe1998}, to  K. Ball's sharp reverse isoperimetric inequality \cite{Ball1991},
to the notion of volume ratio \cite{Szarek78, SzarekTomczakJ80}, which is defined as the $n$-th root of the volume of a convex body divided by the volume of its John ellipsoid
and 
which finds applications in functional analysis and Banach space theory \cite{BourgainMilman87, GPSTW, Schuett1982, SzarekTomczakJ80}.
John and L\"owner position  are even relevant in   quantum information theory \cite{ASW10, ASW11, SWZ11}.
Since a  position may be seen as a choice of a special ellipsoid, 
and since an ellipsoid entails a Euclidean structure of the underlying space,  John and  L\"owner ellipsoids provide a way to measure how far a normed space is from Euclidean space \cite{Gluskin1981, john1948}.
For a detailed discussion of the John and the L\"owner ellipsoid and its connections to functional analysis we refer the reader to \cite{AGM2014, BGVV2014Book, SchneiderBook} and the survey \cite{henkJLellip}.
\par
\noindent
F. John proved in \cite{john1948} that  among all ellipsoids contained in  a convex body \( K\in \R^n\), there is a unique ellipsoid of maximal volume, now called  the John ellipsoid of \(K\). The L\"owner ellipsoid of \(K\) is the unique ellipsoid of  minimal volume containing \(K\). 
These two notions are closely related by polarity (see, e.g., \cite{BGVV2014Book, MSW2015b}): 
A $0$-symmetric ellipsoid $\mathcal{E}$  is the ellipsoid  of maximal volume inside $K$ if and only if $\mathcal{E}^\circ$ is the ellipsoid of
minimal volume outside $K^\circ$, where $K^\circ=\{y \in \mathbb {R}^n: \langle y, x\rangle \leq 1 \  \   \text {for all }\   x \in K\}$
is the polar of $K$.
\par
\noindent
Probabilisitic methods have become extremely useful in convex geometry. In this context, log-concave functions arise naturally from the uniform measure on convex bodies.
A function $f(x)$ is   said to be log-concave, if it is of the form $f(x)= \exp(-\psi (x) )$ where  $\psi: \R^n \rightarrow\R \cup \{ \infty \}$ is convex.
Extensive research has been devoted within the last ten years to extend the concepts and inequalities  from convex bodies to the setting of functions. 
In fact,  it was observed early that the Pr\'ekopa-Leindler inequality (see, e.g., \cite{GardnerBook2006,  PisierBook1989}) is the functional analog of the 
Brunn-Minkowski inequality (see, e.g., \cite{GardnerBulletin}) for convex bodies. Much progress has been made since and 
functional   analogs of many other geometric inequalities were established. Among them are the 
functional Blaschke-Santal\'o inequality \cite{AKM2004, Ball1988, FM2007, Leh2009}  and its reverse \cite{FM2008},  
a functional affine isoperimetric inequality for log-concave functions which can be viewed as an inverse log-Sobolev inequality for entropy  \cite{AKSW2012, CFGLSW}
and a theory of valuations, an important concept for convex bodies (e.g., \cite{Haberl:2012,  Ludwig:2010c, Ludwig:2010, Schuster2010, SchusterWannerer}),  is currently being developed in the functional setting, e.g., \cite{ColesantiLudwigMussnig2017, ColesantiLudwigMussnig, Mussnig2017}.
\par
\noindent
It was only recently that the notion of a John (ellipsoid) function of a log-concave function was established by  Alonso-Guti{\'e}rrez,   Merino, Jim{\'e}nez, 
and Villa \cite{Alonso-Gutiérrez2017}.
However, the notion of  a L\"owner ellipsoid function 
for log-concave functions has been missing till now. In this paper we put forward such a notion
and we investigate, among other things,  its relation to the John ellipsoid function of \cite{Alonso-Gutiérrez2017}.
\par
\noindent
Our main result reads as follows.   We denote by  \( \mathcal A\)  the set of all invertible affine transformations and by $\| \cdot\|_2$ denote the Euclidean norm on $\mathbb{R}^n$.  We say that a function is nondegenerate if $\text{int} (\supp f ) \neq \emptyset$.
\par 
\noindent 
{\bf Theorem. } {\em Let \(f:\R^n\to \R\) be a nondegenerate integrable log-concave function. There exists  a unique pair \((A_0,t_0) \in \mathcal A \times \mathbb{R}\) such that \[\int_{\R^n}e^{-\|A_0x\|_2+t_0}dx=\min\left\{\int_{\R^n}e^{-\|Ax\|_2+t}dx:  t\in \R, A\in \mathcal A, e^{-\|Ax\|_2+t}\ge f(x) \right\}.\]
	The uniqueness of \(A_0\) is up to left orthogonal transformations.}
\vskip 3mm
\noindent
We then call \( e^{-\|A_0x\|_2+t_0}\) the L\"{o}wner function of \(f \) and denote it by 
\[ L (f)(x)=e^{-\|A_0 x\|_2+t_0} .  \]
\par
\noindent
The function  \( L(f)\)  is a functional analog of  the L\"owner ellipsoid for log-concave functions. 
Indeed, we show that  if 
\(\1_K(x)\) is the characteristic function of a  convex body \( K\in \R^n\), then the super-level set \( \{ L(\1_K)\ge 1\}\) is exactly the L\"owner ellipsoid of \(K\). If, in addition, $0$ is the center of the L\"owner ellipsoid of \(K\), then it holds by polarity via the Legendre transform that the polar of the L\"{o}wner function is the John function of $(\1_K)^\circ$. 
This is the exact analog of the above quoted polarity relation of John and L\"owner ellipsoids for a convex body and its polar.
While in the case of convex bodies the two notions of John and L\"owner ellipsoid  are always dual to each other, interestingly, in the functional setting this need no longer be the case.
It holds when the functions are even or characteristic functions of convex bodies.

\vskip 2mm
\noindent
The paper is structured as follows. In Section 2 we introduce the basic facts and preliminaries. In Section 3 we define  the notion of L\"{o}wner  function \(L(f)\) for a log-concave function \(f\) and  we prove its existence and uniqueness. In Section 4,   we recover the John function of \cite{Alonso-Gutiérrez2017} and discuss the  duality between these two notions.

\section{ Notation and preliminaries}

Throughout the paper we will use the following notations.
The set of all non-singular affine transformations on \( \R^n\) is written as \( \mathcal A\), 
\[\mathcal A= \{ A=T+a: T\in GL(n), a\in \R^n\}.\]
Let \( S_+\)  be the set of symmetric positive definite matrices.  Then 
\[\mathcal {SA}= \{ A=T+a: T\in S_+, a\in \R^n\}.\]
For  \( b\in \R^n\) fixed, put 
\[\mathcal A(b)=\{A=T+a: T\in GL(n), a\in \R^n, T^{-1}a=b\}.\]
\par
\noindent
Let \( \mathcal{SA}(b)=\mathcal{A}(b)\cap \mathcal{SA}\). Clearly, \( \mathcal A=\cup_{b\in \R^n} \mathcal A(b)\) and \(\mathcal{SA}=\cup_{b\in \R^n}\mathcal{SA}(b)\).
\par
\noindent
The action of an affine transformation \(A:\R^n\to \R^n\) on a function \(f:\R^n\to \R\) is defined as \( A f(x)=f(Ax)\).
\par
\noindent
For \(z\in \R^n\), let \(S_z\) be a translation of a function  by \(z\), that is, for a function \(f\), 
\be
(S_zf)(x)=f(x+z) \label{defoftranslation}
\ee 
\vskip 2mm
\noindent
For $s \in \R$ and a function   $f: \R^n \rightarrow \R$, we denote by 
$$G_f(s)=\{x\in \R^n: f(x)\ge s\}$$
  the super-level sets of  \( f\).

\subsection{Log-concave functions}

A function $f:  \R^n \rightarrow \R$ is said to be log-concave if it is of the form $f (x) = e^{-\psi(x)}$
where $\psi: \R^n \rightarrow \R \cup \{ \infty \}$ is a    convex function.
We always consider in this paper  log-concave functions $f$  that are integrable and such that $f$ is nondegenerate, i.e., the interior of the support of $f$ is non-empty, $\text{int} (\supp f ) \neq \emptyset$.  
This then implies that $0 < \int_{\R^n} f dx < \infty$.
\par
\noindent
We will also need the Legendre transform which we recall now. Let $z \in \mathbb{R}^n$ and let $\psi:\R^n \rightarrow \R \cup \{ \infty \}$ be a    convex function. Then
\[
\mathcal{L}_z \psi(y)=\sup_{x\in \R^n} [ \langle x-z,y-z\rangle -\psi(x)]
\]
is the Legendre transform of \( \psi\) with respect to $z$  \cite{AKM2004,FM2008} .  If  \( f(x)=e^{-\psi(x)} \) is log-concave, then
\be \label{polar}
f^z(y)=\inf_{x\in \text{supp}( f)}\frac{e^{-\langle x-z,y-z\rangle}}{f(x)}=e^{-\mathcal{L}_z \psi(y)}
\ee
is called the dual or polar function of $f$ with respect to $z$. 
In particular, when \( z=0\),  
$$ f^\circ(y)=\inf_{x\in \text{supp}( f)}\frac{e^{-\langle x,y\rangle}}{f(x)} = e^{-\mathcal{L}_0 \psi(y)},$$
where \( \Lc_0 \), 
 also denoted by \(\Lc\) for simplicity,  is the standard Legendre transform. 
 \par
 \noindent
In the next proposition we collect several  well known, easy to verify,  properties of the generalized Legendre transform that we will use throughout the paper.
They can be found in e.g., \cite{AKM2004} and \cite{FM2007}.

\begin{proposition} \label{propoflegendretransf} Let  $\psi$ be a convex function. 
Let $S_z$ be as in (\ref{defoftranslation}). Then
\par
(i) \( \Lc\) and $ \Lc_z$ are  involutions, that is,  \( \Lc (\Lc \psi)=\psi \) and $ \Lc_z(\Lc_z\psi)=\psi$.
\par	
(ii) $\Lc_z=S_{-z}\circ \Lc\circ S_{z}$.
\par	
	(iii) $ \Lc(S_z\psi)(y)=\Lc\psi-\langle z,y\rangle $.
\par
(iv) Legendre transform reverses the oder relation, i.e., if $\psi_1 \leq \psi_2$, then $\Lc \psi_1 \geq \Lc \psi_2$.
\end{proposition}
\vskip 2mm
\noindent
We now list some basic well-known facts on log-concave functions. A log-concave function is continuous on the interior of its support,  e.g., \cite{RockafellarBook1970}. 

We include a proof  of the first fact for the reader's convenience.  More on log-concave functions can be found in 
e.g., \cite{RockafellarBook1970}. 

\begin{fact} If \(f\) is a nondegenerate integrable log-concave function,  then \( G_f(t)\) is convex and compact for \(0<t \leq \|f\|_\infty\).
\end{fact}
 \begin{proof}
	Let  \( f=e^{-\psi}\).  As $\psi$ is convex and  as $f$ is nondegenerate,  the super-level set 
	\[G_f(t)=\{x:f(x)\ge t\}=\{x: -\psi\ge \log t\}=G_{-\psi}(\log t)\]
	is convex and closed for all $0 < t \leq \|f\|_\infty$. 
As  \( G_f(\|f\|_\infty)\subseteq G_f(t), \forall\ 0<t\le\|f\|_\infty\),	
	it remains to show that \( G_f(t)\) is bounded for  \( 0<t<\|f\|_\infty\). 
	 It follows from  Theorem 7.6 of \cite{RockafellarBook1970} that every super-level set \(G_f(t)\), \( 0<t<\|f\|_\infty\),  has the same affine dimension as the support of \(f\),
	 which has affine dimension \(n\).  
	 Chebyshev inequality  then yields
		\[\vol_n(G_f(t))=\vol_n \left(\left\{ x\in\R^n: f(x)\ge t\right\}\right) \le \frac{\|f\|_1}{t}<\infty.\]
	Since \(G_f(t)\) is a full dimensional convex set with finite volume, it is bounded. 
	Therefore, \( G_f(t)\) is compact for \( 0<t\le\|f\|_\infty\).
\end{proof}
\vskip 2mm
\noindent 
The following fact  is a direct corollary of the functional Blaschke-Santal\'{o} inequality \cite{AKM2004, Ball1988} and the functional reverse Santal\'o inequality \cite{FM2007, KlartagMilman2005}.
\begin{fact}
	Let \( f=e^{-\psi}\) be a nondegenerate, integrable, log-concave function such that $0$ is in the interior of the support of $f$.
	Then $f^\circ$ is again a nondegenerate,  integrable log-concave function and thus 
	  \(0<\int_{\R^n} f^\circ(x)dx<\infty\). Furthermore, 
	  $f^z$ is again a nondegenerate, integrable log-concave function, i.e., \(0<\int_{\R^n} f^z(x)dx<\infty\),  provided that  $z$ is in the interior of $\supp(f)\).
\end{fact}

\section{The L\"{o}wner  function of a log-concave function}

We now define  the L\"owner function for an integrable, nondegenerate, log-concave function \(f=e^{-\psi }\). 
 \subsection{A minimization problem. Definition of the L\"owner function}

We consider the following {\em minimization problem}
\be
\min_{(A,t)}\int_{\R^n}e^{-\|Ax\|_2+t}dx 
\label{minimalprob}
\ee
subject to 
\be \|Ax\|_2-t\le \psi (x),   \label{subjcond} \  \   \text{ for all  } x \in \mathbb{R}^n, \ee
where  the minimum is taken over all nonsingular affine maps \( A\in \mathcal{A}\) and all \(t\in \R\). 
A change of variables leads to 
\begin{eqnarray*}
\min_{(A,t)}\int_{\R^n}e^{-\|Ax\|_2+t}dx &=&\min_{(A,t)}e^t\int_{\R^n}e^{-\|Ax\|_2}dx
=\min_{(A,t)}\frac{e^t}{|\det A|}\int_{\R^n}e^{-\|y\|_2}dy \\
&=& n! \  \vol(B_2^n) \  \min_{(A,t)}\frac{e^t}{|\det A|}.
\end{eqnarray*}
Geometrically  this  means that we minimize the integral of an {\em ellipsoidal} function $e^{-\|Ax\|_2+t}$ ``outside" $f$ which is exactly what is done when
one considers the L\"{o}wner ellipsoid of a convex body $K$: it minimizes  the volume of the ellipsoids containing $K$.
\par
\noindent
The next theorem is the main result of this section. 

\noindent
\begin{theorem} \label{lownerfcn} 
Let \(f: \R^n\to \R^+\), \(f(x)=e^{-\psi(x)}\) be a nondegenerate,  integrable   log-concave function.
Then there exists a unique  solution modulo $O(n)$ to the minimization problem (\ref{minimalprob}) and (\ref{subjcond}). That is, there exists a  pair \((A_0,t_0)\) satisfying (\ref{subjcond}) such that 
\[\min_{(A,t)}\int_{\R^n}e^{-\|Ax\|_2+t}dx =  n! \  \vol(B_2^n) \  \frac{e^{t_0}}{|\det A_0|}.\]
The number $t_0$ is unique and the affine map $A_0$ 
 is unique  up to left orthogonal transformations.
\par
\noindent
We then call \( e^{-\|A_0x\|_2+t_0}\) the L\"{o}wner function of \(f \) and denote it by 
\[ L(f)(x)=e^{-\|A_0 x\|_2+t_0} .  \]
\end{theorem}
\vskip 2mm
\noindent
{\bf Examples.}
\par
\noindent
1. The L\"{o}wner function  is an extension of the concept of L\"{o}wner ellipsoid for convex bodies. Indeed, let
\begin{eqnarray*}
		f(x)=\1_{K}(x) = e^{- I_K(x)},  \ \text{ where } \  
		I_K(x) = 
		 \begin{cases} 
			\infty & x\notin K \\
			0 & x\in K 
		\end{cases}\\ 
\end{eqnarray*}
be  the characteristic function of a convex set $K \subset \mathbb{R}^n$. Without  loss of generality we may assume that $0$ is the center of the 
 L\"{o}wner ellipsoid $L(K)$ of $K$.  Then 
\begin{equation}\label{LownerK}
 L (\1_{K})(x) =  e^{-n \left(\|T_{L(K)}^{-1} x\|_2 -1\right)},
 \end{equation}
where $T_{L(K)}$ is the linear map such that $T_{L(K)} B^n_2= L(K)$.
To see this, observe that for  $A \in \mathcal{A}$, $t \in \mathbb{R}$,  the level sets of the map $\varphi(x) =  \|Ax\|_2-t$ are ellipsoids. As $0$ is the center of the 
 L\"{o}wner ellipsoid of $K$, $A=T+a$ is such that $a=0$. 
Thus we get in particular, that the level set
$$
\{x: \varphi(x) = 0\} = \{x: \|Tx\|_2= t\} = t \  T^{-1} B^n_2.
$$
As we require that $\|Ax\|_2-t \leq 0$, for all $x \in K$, the smallest ellipsoid that satisfies this is the  L\"{o}wner ellipsoid $L(K)$ of $K$, i.e., $ t \  T^{-1} B^n_2 = L(K)$. Thus
$$
|\det T | = t^n \  \frac{\vol_n(B^n_2)}{\vol_n(L(K))} 
$$
and $ \min_{(T,t)}\frac{e^t}{|\det T|}$ is achieved for $t_0= n$. This means that $T_0=n T_{L(K)}^{-1}$ and hence $ L (\1_{K})(x) =  e^{-n \left(\|T_{L(K)}^{-1} x\|_2 -1\right)}$.
\par
\noindent
2. It is easy to see that the L\"{o}wner function of the Gaussian $g(x) = e^{-\|x\|^2_2/2}$ is given by 
$$L (g)(x)= e^{-\sqrt{n} \|x\|_2 +\frac{n}{2}}.$$
\par
\noindent
3. More generally, let $f(x) = e^{-\psi(x)}$ be a log-concave function where the convex function $\psi$ depends only on the Euclidean norm of $x$, 
$\psi(x) = \varphi(\|x\|)$. Then  by symmetry $A \in \mathcal{A}$ is of the form $A_0= a \  \text{Id}$. We compute that $a$ and $t_0$ are determined by
$$a = \varphi^\prime\left(\frac{n}{a}\right),   \hskip 2mm t_0= n - \varphi\left(\frac{n}{a}\right)$$
and thus 
$$
L(f)(x) = e^{- a \|x\|_2 +n - \varphi\left(\frac{n}{a}\right)}.
$$

\vskip 4mm
\noindent
We will prove  Theorem \ref{lownerfcn} in several steps.  The first one is to give an equivalent simplified version of   the  minimization  problem via a reduction argument.

\subsection{A reduction argument } \label{Redux}

Let \( f=e^{-\psi}\) be a log-concave function.   Let $A=T+a  \in \mathcal{A}$. 
By the polar decomposition theorem,  \( T\in GL(n)\) can be written as  \( T=O \ R\),  where  \( R\) is a symmetric positive definite matrix and \( O\in O(n)\), the set of orthogonal matrices.
Then
\begin{eqnarray*}
&\min_
{
\|Ax\|_2\le \psi (x)+t
}
 \frac{e^t}{|\det A|}\hskip 2mm  = \hskip 2mm    \min_
{
\|Tx +a\|_2\le \psi (x)+t
} \frac{e^t}{|\det T|}  = \nonumber \\  
&\min_
{
\|O R x + a\|_2\le \psi (x)+t 
} \frac{e^t}{\det R} \hskip 2mm  = \hskip 2mm  \min_
{
\| R x + O^t a\|_2\le \psi (x)+t
} \frac{e^t}{\det R}  \nonumber  \\
&=   \hskip  2mm \min_
{
\| Ax\|_2\le \psi (x)+t
} 
\frac{e^t}{\det A} ,  
\end{eqnarray*}
where $A   \in \mathcal{S A}$. 
 Thus we may assume that $A= T+a$, where  \( T\) is symmetric and positive definite, i.e., $T \in S_+$.  We put  $b=T^{-1} a$ 
 and re-write the last expression further.
\begin{eqnarray} \label{equivalent2}
&& \hskip -5mm \min_
{\| Ax\|_2\le \psi (x)+t } \frac{e^t}{\det A}   \hskip 2mm   = \left( \hskip 2mm  \max_
{
\| Ax\|_2\le \psi (x)+t
 } e^{-t} \  \det A \right)^{-1} \nonumber \\
&&
 = \left(\max_{t }  \hskip 1mm    \max_
{
\|Ax\|_2\le \psi (x)+t
} e^{-t} \  \det A \right)^{-1}  
 =  \left( \max_{t }  \hskip 1mm    \max_
{
\|Tx +a\|_2\le \psi (x)+t
 } e^{-t} \  \det T\right)^{-1} \nonumber \\ && 
 = \left(  \max_{t }  \hskip 1mm    \max_
{
\|T(x +b)\|_2\le \psi (x)+t
 } e^{-t} \  \det T \right)^{-1} 
 =  
\left( \max_{t \in \R }    \hskip 1mm     \max_{b \in \R^n}  \hskip 1mm    \max_
{\|Tx \|_2\le \psi (x-b)+t}
 e^{-t} \  \det T \right)^{-1} \nonumber \\
&& 
= \left(  \max_{b \in \R^n }    \hskip 1mm     \max_{t \in \R}  \hskip 1mm    \max_
{
\|Tx \|_2\le \psi (x-b)+t
} e^{-t} \  \det T\right) ^{-1} \nonumber \\
&&= \min_{b \in \R^n }    \hskip 1mm   \left(       \max_{t \in \R}  \hskip 1mm    \max_
{
\|Tx \|_2\le \psi (x-b)+t
} e^{-t} \  \det T\right) ^{-1},
\end{eqnarray}
where $T \in S_+$.
\vskip 2mm
\noindent
This leads us to  first consider an optimization problem  for  fixed  $b \in \mathbb{R}^n$.
\par
\noindent
\begin{proposition} \label{reducedprob}
Fix \( b\in \R^n\). Let \( f=e^{-\psi}\)  be a nondegenerate,  integrable  log-concave function on \(\R^n\). 
There exists a unique solution,   up to left orthogonal transformations,  to the maximization  problem 
\begin{eqnarray}
   \max_{T \in S_+, t\in \R^n} e^{-t} \  \det T  \hskip 4mm \text{subject to } \hskip 4mm   \|Tx\|_2-t\le \psi (x-b)\hskip 2mm \forall x\in\mathbb R^{n}. 
\label{restrictminimalprob}
\end{eqnarray}
\end{proposition}
\par
\noindent
\vskip 2mm
\noindent
Before we prove Proposition \ref{reducedprob},  we re-write the constraint condition of  (\ref{restrictminimalprob}).  
\par
\noindent
For any function $h:\mathbb R^{n}\to\mathbb R$ we define its diametral with respect
to the point $w$ as 
$$
h_{\operatorname{dia},w}(-x+2w)=h(x).
$$
For a convex function $\psi:\mathbb R^{n}\to\mathbb R$ we define its symmetral 
$\psi_{\operatorname{sym},w}$ with
respect to the point $w$ as the greatest, convex function that is smaller than
$\max\{\psi,\psi_{\operatorname{dia},w}\}$.
In the same way we define the symmetral 
$f_{\operatorname{sym},w}=e^{-\psi_{\operatorname{sym},w}}$
of a log-concave function $f=e^{-\psi}$.
\par
Since for all $x\in\mathbb R^{n}$
$$
\|T(x+b)\|_2-t=\|T((-x-2b)+b)\|_2-t
$$
the condition
$$
\forall x\in\mathbb R^{n}:\hskip 3mm\|T(x+b)\|_2-t\leq \psi(x)
$$
is equivalent to the condition
$$
\forall x\in\mathbb R^{n}:\hskip 3mm\|T(x+b)\|_2-t\leq \psi_{\operatorname{sym},-b}(x).
$$
Therefore, we can assume that the convex function $\psi$ is symmetric with respect to $-b$.
By Proposition \ref{propoflegendretransf} and Fact 2, 
taking the  Legendre transform on both sides  yields the equivalent condition
\be \mathcal{L}(\|Tx\|_2-t)(y)\ge  \mathcal{L} \left( \psi (x-b)\right) (y) = \mathcal{L}\circ S_{-b} \psi(y). \label{legendrerefinedrestrictsubjcond}
\ee
Observe that 
\begin{eqnarray*}
	\Lc(\|Tx\|_2-t)(y)&=& \sup_x \langle x,y\rangle-\|Tx\|_2+t
	= t+\sup_x \langle x,y\rangle-\|Tx\|_2\\
	&=& t+\sup_z \langle T^{-1}z,y\rangle- \|z\|_2 
	= t+\sup_z \langle z, (T^{-1})^t y\rangle -\|z\|_2\\
	&=& t+\begin{cases} 
		\infty & \|(T^{-1})^t y\|_2>1 \\
		0 & \|(T^{-1})^t y\|_2\le 1 
	\end{cases}\\
	&=& t+\begin{cases} 
		\infty & y\notin T  B_2^n \\
		0 & y\in T B_2^n , 
	\end{cases}\\
\end{eqnarray*}
where from the second to the third equality we have put $ z=Tx$.   It follows that 
\[
e^{-\mathcal{L}(\|Tx\|_2-t)(y)}=e^{-t}\1_{T B_2^n}.
\]
If we set \( f_b=S_{-b}f\),  then 
  (\ref{legendrerefinedrestrictsubjcond}) is equivalent to 
\[
e^{-t}\1_{T^tB_2^n}\le (f_b)^\circ.
\]
Note that  by Fact 2,  \(( f_b)^\circ\) is an integrable log-concave function, provided 
$b \in \text{int } ( \supp f)$.
When $b \notin \text{int } ( \supp f)$, we replace $f$ by $f_{\text{sym,-b}}$  and by the above considerations the minimization problem remains the same.
\par
\noindent
Moreover, shifting by a vector $b$ does not affect the existence and uniqueness of the solution to the optimization problem in Proposition \ref{reducedprob} and  hence proving Proposition \ref{reducedprob} is equivalent to proving the case   $b=0 \in \text{int } ( \supp f)$, possibly replacing $f$ by $f_{\text{sym}}$, i.e., 
we need to show that there is a unique solution modulo $O(n)$  to the maximization problem
\be\label{restriction2}
    \max_{T\in S_+,t\in \R} e^{-t}\det T
	\hskip 4mm \text{subject to } \hskip 4mm
	e^{-t}\1_{TB_2^n}\le f ^\circ. 
		\ee
By Proposition \ref{propoflegendretransf} and the Fact 2,   to prove (\ref{restriction2}),  and hence Proposition \ref{reducedprob},  it is enough to  prove the following Proposition.
\vskip 2mm
\noindent
\begin{proposition}\label{equivreducedprob}
	Let \( f=e^{-\psi}\) be a nondegenerate,  integrable  log-concave function.
	Then
	there exists a  unique solution $(t_0, T_0) \in \mathbb{R} \times S_+$, up to right orthogonal transformations,   to the maximization problem 
	\be
    \max_{T\in S_+,t\in \R} e^{-t}\det T
	\hskip 4mm \text{subject to } \hskip 4mm 
	e^{-t}\1_{TB_2^n}\le f .   \label{restrictsubjcond}
	\ee
\end{proposition}
\vskip 4mm
\noindent

\subsection {Proof of Proposition \ref{equivreducedprob}}

To prove Proposition \ref{equivreducedprob},  we introduce,  for \( 0<s\le \|f\|_\infty\),
\[
 \xi_f(s):= s \max_{\{T\in S_+:TB_2^n\subset G_f(s)\}}   \det T.
 \]
 Then we can re-write (\ref{restrictsubjcond}) in terms of $\xi_f$, namely, 
 \begin{equation}\label{Gleichung}
   \max \{ e^{-t} \ \det T : \   T\in S_+,  \  t\in \R, \  e^{-t}\1_{TB_2^n} \  \le \  f\} =\max_{0<s\le \|f\|_\infty}\xi_f(s). 
    \end{equation}
     Indeed, putting $s = e^{-t}$, 
 \begin{eqnarray*}
&&	\hskip -10 mm \max \{ e^{-t} \ \det T : \   T\in S_+,  \  t\in \R, \  e^{-t}\1_{TB_2^n} \  \le \  f\} = \\
	&& \hskip 15mm \max\{ s\  \det T : \  T\in S_+,\   s>0, \  s\1_{TB_2^n} \  \le \  f\} . \\ 
\end{eqnarray*}
Note that \(s\1_{TB_2^n}\le f\Longleftrightarrow TB_2^n\subset G_f(s) \).
Thus we may restrict our attention to the set  
\[
\cup_{s>0} \{T \in S_+: \   TB_2^n\subset G_f(s)\}
 .\]
When \(s>\|f\|_\infty\), \(\{T: \   TB_2^n\subset G_f(s)\}=\emptyset \).
Thus we consider 
\[ \bigcup_{0<s} \{T \in S_+: \  TB_2^n\subset G_f(s)\} =\bigcup_{0<s\le \|f\|_\infty } \{T \in  S_+: \  TB_2^n\subset G_f(s)\}. \]
Therefore, 
\begin{eqnarray}\label{xi-Gleichung}	
	\max\{ s\det T:  T \in S_+, s>0, s \1_{TB_2^n}\le f\} 
	&=&\max_{0<s\le \|f\|_\infty} s \max _{\{T\in S_+: \   TB_2^n\subset G_f(s)  \}} \det T  \nonumber\\
	&= &\max_{0<s\le \|f\|_\infty}  \xi_f(s)
\end{eqnarray}
 \vskip 4mm
 \noindent
 We shall show in  the next lemma that \( \lim_{s\to 0} \xi_f(s)=0\) and in Corollary \ref{continuityofxi} below that the map \(s\to \xi_f(s)\) is continuous.  
We   then can  conclude that the maximizer in Proposition \ref{equivreducedprob} exists.
\vskip 2mm
\noindent
 The next lemma and its proof is similar to Lemma  2.1  in \cite{Alonso-Gutiérrez2017}. We include a proof  for completeness. 
\begin{lemma}\label{propofxi} Let \( f=e^{-\psi}\) be an integrable,  nondegenerate, log-concave function on \(\R^n\).
	For any \( s_1,s_2\in (0,\|f\|_\infty]\) and \( 0\le \lambda\le 1\), 
	\be
	\xi_f(s_1^{1-\lambda}s_2^\lambda)\ge\xi_f(s_1)^{1-\lambda}\xi_f(s_2)^\lambda.
	\label{inequalityofxi}
	\ee
	Moreover, \( \lim_{s\to 0} \xi_f(s)=0\).
\end{lemma}
\par
\noindent
\begin{proof} As the set $\{T\in S_+:TB_2^n\subset G_f(s)\}$ is compact (e.g., in the operator topology), and as the determinant is continuous, there are 
	 \( T_0, T_1\) and \( T_2\)  such that \(\xi_f(s_1^{1-\lambda}s_2^\lambda)=s_1^{1-\lambda}s_2^\lambda\cdot \det T_0,\  \xi_f(s_1)=s_1\cdot \det T_1\) and 
	\( \xi_f(s_2)=s_2\cdot \det T_2\). 
	Then, as $f$ is log-concave,   
		\begin{eqnarray*}
	G_f(s_1^{1-\lambda}s_2^\lambda)&=&\{x: f(x)\ge s_1^{1-\lambda}s_2^\lambda\}
	\supset (1-\lambda)\{x: f(x)\ge s_1\}+\lambda \{x:f(x)\ge s_2\}\\
	&=& (1-\lambda) G_f(s_1)+\lambda G_f(s_2)
	\supset (1-\lambda)T_1B_2^n+\lambda T_2B_2^n\\
	&\supset&((1-\lambda)T_1+\lambda T_2)B_2^n.	
	\end{eqnarray*}
    Hence \( \det T_0 \ge \det[(1-\lambda)T_1+\lambda T_2)]\).    
    Moreover, we have \( \det T_0\ge (\det T_1)^{1-\lambda} (\det T_2)^\lambda\). Indeed, by Minkowski's determinant inequality for positive definite matrices (see, e.g., \cite{RobertsVarbergBook1973}), 
    \begin{eqnarray}
    \det T_0&\ge & \det [(1-\lambda)T_1+\lambda T_2) ]  \nonumber\\
              &\ge & \left( (1-\lambda)(\det T_1)^{1/n}+\lambda (\det T_2)^{1/n} \right) ^n  \label{minkowskidetineq}
              \\
              &\ge & (\det T_1)^{1-\lambda}(\det T_2)^\lambda .\label{agmeanineq}
    \end{eqnarray}
    The last inequality follows from the arithmetic-geometric mean inequality. 
    Therefore, 
    \[s_1^{1-\lambda}s_2^\lambda \det T_0 \ge \left( s_1\det T_1 \right)^{1-\lambda}\left( s_2\det T_2)\right) ^\lambda. \]    
    In \cite{Alonso-Gutiérrez2017}, the authors introduce, for $t >0$,  a function \( \phi_f(t)\), 
    \[ 
    \phi_f(t)=\max_{\{A\in \mathcal{A}:  \  AB_2^n\subset G_f(t)\}}t\cdot | \det A |.
    \]
    They showed that \(\lim_{t\to 0}\phi_f(t)=0\). It is clear that \( \xi_f(s)\le \phi_f(s)\) for all \(s\). Hence \(\lim_{s\to 0}\xi_f(s)=0\).
\end{proof}
\vskip 3mm
\noindent
Next we state a John-type result which is  well known. We include a proof for completeness.  We recall the Hausdorff metric, which for two convex bodies $K$ and $L$ is defined as 
$$
d_H(K,L)
= \min\{\lambda \geq 0: K \subseteq L+\lambda B^n_2;   L \subseteq K+\lambda B^n_2 \} .$$
\begin{lemma} \label{contiofellipsoid}
	Let \( \mathcal{K}^n\) be the set of convex bodies in \( \R^n\), equipped with the Hausdorff metric.
	The  map \[K \to  \max_{\{T\in S_+: \  TB_2^n\subset K\}}\det T \] is continuous in \( K\). 
	Moreover, let \(T_K\) be a maximizer, i.e., 
	\[\det T_K =\max_{\{T\in S_+:  \   TB_2^n\subset K\}}\det T .  \]
	Then $T$ is unique up to an orthogonal transformation.
	\end{lemma} 
\begin{proof} 
	First note that if \( 0\notin int(K)\), then $\{T \in S_+: \  TB_2^n\subset K\} = \emptyset $.\\
	For \(K\) with \( 0\in int(K)\), let \(T_K\) be such that \( \det T_K= \max_{\{T\in S_+: TB_2^n\subset K\}}\det T\) and let \( \hat{K}=K\cap (-K)\). Then 
	\[ T_KB_2^n\subset \hat{K}=K\cap (-K)\subset K.\]
As $K \cap (-K)$ is centrally symmetric, the center of the ellipsoid of maximal volume contained in 	$K \cap (-K)$ is also centered at $0$.
	Therefore the  ellipsoid \( T_KB_2^n\) is the  ellipsoid of largest volume or John ellipsoid \(J(\hat{K})\) contained in \(\hat{K}=K\cap (-K)\).
	It follows that \(T_K\) is unique, modulo $O(n)$, as \(J(\hat{K})\) is unique, e.g.,  \cite{GardnerBook2006}.
	\par
	\noindent	
	Now notice that if   \( K\) and $L$ are such that \( d_H(K,L)<\delta\), then \( d_H(\hat{K}, \hat{L})< 2\delta\). In fact, on the one hand, \[ \hat{L}\subset L\subset K+\delta B_2^n\] 
	\[\hat{L}\subset -L\subset -K+\delta B_2^n,\]
	hence
	\[ \hat{L}\subset K\cap (-K)+2\delta B_2^n =\hat{ K}+2\delta B_2^n.\]
	The other direction follows similarly. 
Let $K \in \mathcal {K}^n$. The map $\hat{K} \rightarrow 	J(\hat{K})$ is continuous, see e.g., \cite{Gruenbaum1963}. Hence, for all $\varepsilon >0$ there exists $\delta$ such that for all $L\in \mathcal {K}^n$ with \( d_H(\hat{K}, \hat{L})< \delta\) we have \( d_H(J(\hat{K}), J(\hat{L}))< \varepsilon\).  It follows that for all $L$ with \( d_H(K,L)<\delta/2\), we get 
\[d_H(T_KB_2^n, T_LB_2^n)<\eps.\]
\end{proof}

\begin{corollary} \label{continuityofxi}
	The map \(s\to  \xi_f(s)\) is continuous in \(s\). 
\end{corollary}
\begin{proof}
	Note that the map 
	\[ s \to \max_{\{T\in S_+ : TB_2^n\subset G_f(s)\}} \det T \] 
	is continuous in \(s\) as it is the composition of  the continuous maps \( s\to G_f(s)\) and \( K\to \max_{\{T\in S_+: TB_2^n\subset K\}}\det T\). 
	Hence, 
	\[ s\to  s \cdot \max_{\{T\in S_+: TB_2^n\subset G_f(s)\}} \det T =  \xi_f(s)\] is continuous in \(s\). 
\end{proof} 
\vskip 3mm
\noindent
Now we are ready for the  proof of  Proposition \ref{equivreducedprob} .
\begin{proof}
As  \( \lim_{s\to 0}\xi_f(s)=0\) by Lemma \ref{propofxi}, and as \(\xi_f(s)\) is continuous on \( (0,\|f\|_\infty]\),
\( \xi_f(s)\) attains its maximum for some \( s_0\in (0,\|f\|_\infty]\) and \( T_0\in S_+\). In other words, \( t_0=-\log s_0\) and \( T_0\) solve the maximization problem in Proposition \ref{equivreducedprob}. 
To see the uniqueness modulo $O(n)$, it suffices to show uniqueness in \(s\).  Uniqueness  in \(T\) modulo $O(n)$ then follows from Lemma \ref{contiofellipsoid}.  
\par
\noindent
Suppose there are \( s_1,s_2\) such that \( s_1>s_2\) and  \(\xi_f(s_1)=\xi_f(s_2)\).  Then  it follows from (\ref{inequalityofxi}) and the definition of $\xi_f$  that for \( 0\le \lambda\le 1\),
\[ \xi_f(s_1^{1-\lambda}s_2^\lambda)=\xi_f(s_1)^{1-\lambda}\xi_f(s_2)^\lambda.\]
As in the proof of Lemma \ref{propofxi}, let  \( T_0, T_1\) and \( T_2\) be such that 
\[\xi_f(s_1^{1-\lambda}s_2^\lambda)=s_1^{1-\lambda}s_2^\lambda\cdot  \det T_0 ,\hskip 3mm  \xi_f(s_1)=s_1\cdot \det  T_1,  \hskip 3mm  
 \xi_f(s_2)=s_2\cdot \det T_2 .\]  Then 
\[\det T_0=(\det T_1)^{1-\lambda}(\det T_2)^\lambda.\]
In other words, we have equality in the Minkowski determinant inequality and in the arithmetic-geometric mean inequality,  (\ref{minkowskidetineq}) and (\ref{agmeanineq}), which implies that  \( \det T_1=\det T_2\). Thus
\[ \xi_f(s_1)=s_1\det T_1=s_1\det T_2>s_2\det T_2=\xi_f(s_2), \]
which is contradiction.  
\end{proof}
\vskip 4mm
\noindent

\subsection{Proof of Theorem \ref{lownerfcn}} 
We need several more lemmas. Some of  them are well known. We include a proof for the reader's convenience.

\begin{lemma} \label{ptwcontoflevelset}
	
	Let \(\{f_m\},f\) be nondegenerate integrable log-concave functions such that  \( f_m\to f\) pointwise.  Then the super-level sets converge in Hausdorff metric, that is,
	\[G_{f_m}(k)\to G_f(k) \text{ in Hausdorff,  for  }0< k < \|f\|_\infty  .\]
	
\end{lemma} 

\begin{proof}	
 Since \( f_m,f\) are non-degenerate, integrable log-concave functions,  they are continuous on their support and by Fact 1, \( G_f(k)\) is a convex body for \( 0<k < \|f\|_\infty\) and \(G_{f_m}(k)\) is a convex body for \( 0<k <  \|f_m\|_\infty\) and  all \(m\geq 1\).
 \par 
 \noindent 
We fix $k$. By e.g., Theorem 1.8.8 of \cite{SchneiderBook}, convergence of  $G_{f_m}(k)\to G_f(k)$  in the Hausdorff metric is equivalent to  the following two properties to hold:
\newline
(i) the limit of any convergent subsequence $(x_{m_j})_{j \in \mathbb N}$ with $x_{m_j} \in G_{f_{m_j}}(k)$ for all $j$, belongs to $G_{f}(k)$;
\newline
(ii) each point in $G_{f}(k)$ is the limit of a sequence $(x_{m})_{\in \mathbb N}$ with $x_{m} \in G_{f_m}(k)$ for all $m \in \mathbb N$.
\par
\noindent
We show (i). Let $(x_{m_j})_{j \in \mathbb N}$ be a sequence with $x_{m_j} \in G_{f_{m_j}}(k)$ for all $j$ and let $x=\lim_{j \to \infty} x_{m_j} $. 
Let $D = \overline{\text{co}}[\{x_{m_j} : j \in \mathbb{N} \}]$ be the closed convex hull of $\{x_{m_j} : j \in \mathbb{N} \}$.  Then $D$ is compact  and convex and as $f_{m_j} \to f$ pointwise on $\mathbb R^n$,  $f_{m_j} \to f$ uniformly
on $D$, by e.g., Theorem 10.8 of \cite{RockafellarBook1970}. Therefore, for $j$ large enough,
\begin{equation}\label{claim}
| f_{m_j} (x_{m_j}) - f(x)| \leq | f_{m_j} (x_{m_j}) - f(x_{m_j})| + | f (x_{m_j}) - f(x)|  < 2 \varepsilon.
\end{equation}
The first estimate holds by the uniform convergence and the second by continuity of $f$. Inequality (\ref{claim}) says exactly that  $ f_{m_j} (x_{m_j}) \to f(x)$.
As $ f_{m_j} (x_{m_j}) \geq k$, we thus get that $f(x) \geq k$ and hence $x \in G_{f}(k)$.
\par
\noindent
Now we show (ii).  By definition, for $0< k < \|f\|_\infty$, 
$$
 G_{f}(k) = \{x: f(x) \geq k\} =\{x: \psi(x) \leq - \log k\} = E_\psi(l),
 $$
where we have put $l= -\log k$.  Similarly, we rewrite $G_{f_m}(k)= E_{\psi_m}(l)$ and then  need to show that every $x \in E_{\psi}(l)$  is the limit of a sequence  
$(x_{m})_{\in \mathbb N}$ with $x_{m} \in E_{f_m}(k)$ for all $m$.  We can assume that $\psi(x)=l$. As  $f$ is integrable, there is $x_0$ in $\mathbb{R}^n$ such that $\psi(x_0)= \min_{x \in \mathbb{R}^n} \psi(x)$. 
We assume without loss of generality that $x_0=0$ and consider the $2$-dimensional plane spanned by $x$ and $e_{n+1} = (0, \dots, 1)$. As $k < \|f\|_\infty$, $l > \psi(x_0) = \psi(0)$. 
Let  $0 < 2 \varepsilon ^\frac{1}{2} < \psi(x) - \psi (0)$.  As $f_m \to f$ pointwise, $\psi_m \to \psi $ pointwise and therefore we have for all $m \geq m_0$  that 
\begin{equation*} 
|\psi(x) - \psi_m(x) | < \varepsilon \hskip 5mm \text{and}  \hskip 5mm |\psi(0) - \psi_m(0) | < \varepsilon.
\end{equation*}
Let $L$ be the line determined by $(0,  \psi(0)+ \varepsilon)$ and $(x, \psi_m(x))$ and let 
$$
x_m = \frac{l -(\psi(0) + \varepsilon)}{\psi_m(x) - (\psi(0) + \varepsilon)} \  x, 
$$
that is $x_m$ is such that the value of $L$ at $x_m$ is $l$.
Then
\begin{eqnarray*}
\|x_m -x\|_2 = \|x\|_2 \   \frac{| l - \psi_m(x)|}{|\psi_m(x) - (\psi(0) + \varepsilon)|} \leq \frac{\varepsilon}{|\psi_m(x) - (\psi(0) + \varepsilon)|} \leq \frac{\varepsilon^\frac{1}{2}}{2(1-\varepsilon^\frac{1}{2})}.
\end{eqnarray*}
The last inequality holds as $|\psi_m(x) - (\psi(0) + \varepsilon)| = |\psi_m(x) - \psi(0) - \varepsilon| > 2  \varepsilon^\frac{1}{2} - 2 \varepsilon$. 
By convexity of $\psi_m$ we have for all $y$ in  the line segment $[0,x]$ that $\psi_m(y)
 \leq L(y)$. 
If $ \psi_m(x ) \geq \psi(x) $ for all $m \geq m_0$, then $x_m \in [0,x]$ and thus
$$
\psi_m(x_m) \leq L(x_m)  \leq l,
$$
which means that $x_m \in E_{\psi_m}(l)$ and we are done.
If there exists $m_1 \geq m_0$ such that $\psi_{m_1}(x) < \psi(x) =l$, then $x \in E_{\psi_{m_1}}(l)$ and we take $x_{m_1} =x$.  Thus, for all $m > m_1$, either 
$\psi_{m}(x) \geq  \psi(x)$  and then we put $x_m$ as above or $\psi_{m}(x) < \psi(x)$ and then we put $x_m=x$.
\end{proof}

\vskip 2mm
\noindent
\begin{lemma}\label{convergenceinfnorm}
	Let \(\{f_m\},f\) be nondegenerate integrable log-concave functions such that  \( f_m\to f\) pointwise.  Then	
	\( \|f_m\|_\infty \to \|f\|_\infty\).
	
\end{lemma} 

\begin{proof}
	As $f$ is integrable and log-concave, there is \(x_0\in \R^n\)  such that \(f(x_0)=\|f\|_\infty \). 
	Thus
	for an arbitrary \(\eps>0\), there exists \(m_1\) such that \[f_m(x_0)\ge f(x_0)-\eps , \]
	whenever \(m>m_1\). So \( \|f_m\|_\infty\ge f_m(x_0)\ge f(x_0)-\eps \) whenever \( m>m_1\). Thus 
	\be \label{polarbdliminf}
	\liminf\|f_m\|_\infty\ge \|f\|_\infty.
	\ee
	On the other hand, fix an arbitrary \(0 < \eps < \frac{1}{4} \|f\|_\infty \). By   log-concavity
	of \( f\), there exists \( \delta>0\) such that 
	\[G_{f}\left(\frac{1}{2}\|f\|_\infty-\eps \right)\subset G_{f}\left(\frac{1}{2}\|f\|_\infty\right)+\delta B_2^n. \]
	By  Lemma \ref{ptwcontoflevelset}, there exists \(m_2\) such that 
	\be
	G_{f_m}\left(\frac{1}{2}\|f\|_\infty \right)\subset G_{f}\left(\frac{1}{2}\|f\|_\infty\right)+\delta B_2^n, 
	\ee
	whenever \(m>m_2\). It follows that 
	\(
	f_m(x)<\frac{1}{2}\|f\|_\infty
	\)
	for all \(x\notin G_{f}\left(\frac{1}{2}\|f\|_\infty\right)+\delta B_2^n \) and whenever \(m>m_2\). In other words, 
	\be \label{polarbdout}
	\sup_{x\notin G_{f}\left(\frac{1}{2}\|f\|_\infty\right)+\delta B_2^n}f_m(x) \leq \frac{1}{2}\|f\|_\infty, 
	\ee
	whenever \(m>m_2\).
	Moreover, since \( f_m(x)\to f(x)\) pointwise on \(G_{f}\left(\frac{1}{2}\|f\|_\infty\right)+\delta B_2^n \) and 
	\(G_{f}\left(\frac{1}{2}\|f\|_\infty\right)+\delta B_2^n\) is a compact set, we have \( f_m\to f\) uniformly on \(G_{f}\left(\frac{1}{2}\|f\|_\infty\right)+\delta B_2^n\), 
	by e.g., Theorem 10.8 of \cite{RockafellarBook1970}.  That is, for the same \(\eps\), there exists \( m_3\) such that 
	\[ f_m(x)\le f(x)+\eps\]
	whenever \(m>m_3\) and for all \( x\in G_{f}\left(\frac{1}{2}\|f\|_\infty\right)+\delta B_2^n\).
	Thus, 
	\be \label{polarbdin}
	\sup_{x\in G_{f}\left(\frac{1}{2}\|f\|_\infty\right)+\delta B_2^n}
	f_m(x)\le \|f\|_\infty+\eps, 
	\ee
	whenever \(m>m_3\). Taking \(m>\max\{m_2,m_3\}\) and combining (\ref{polarbdout}) and (\ref{polarbdin}), one has 
	\[\sup_{x\in \R^n} f_m(x)=\|f_m\|_\infty\le \|f\|_\infty+\eps.\]
	Hence 
	\be \label{polarbdlimsup}
	\limsup\|f_m\|_\infty\le \|f\|_\infty.
	\ee	
	Finally, combining (\ref{polarbdliminf}) and (\ref{polarbdlimsup}), one concludes that \(\lim\|f_m\|_\infty=\|f\|_\infty\).
\end{proof}
\vskip 2mm
\noindent
\begin{lemma} \label{continuityI(b)}
Let \(\{f_m\},f\) be a nondegenerate integrable log-concave functions  and suppose that  \( f_m\to f\) pointwise.  Then 
\[\max_{0<s\le \|f_m\|_\infty}\xi_{f_m}(s)\to \max_{0<s\le \|f\|_\infty}\xi_f(s).\] 
\end{lemma}
\vskip 2mm
\noindent
\begin{proof} 
For $m \geq 0$,  and with the convention that $f_0=f$, let \( T_{m,s}\) be such that \[ \det T_{m,s}=\max _{\{T\in S_+: \   TB_2^n\subset G_{f_m}(s)  \}} \det T.\]
By (\ref{Gleichung}) and Proposition \ref{equivreducedprob}, there exists a unique $s_0 = e^{-t_0}$ and a unique, modulo $O(n)$,
$T_{0} \in S_+$  such that 
\begin{eqnarray*}\label{}
\xi_{f}(s_0) &=& \max_{0 < s \leq \|f\|_\infty} \xi_{f}(s) = s_0  \det T_{0}  =  \max_{0 < s \leq \|f\|_\infty} s \  \max _{\{T\in S_+: \   TB_2^n\subset G_{f}(s)  \}}  \det T\nonumber  \\ 
&= & \max_{0 < s \leq \|f\|_\infty} s \det T_{0,s}.
\end{eqnarray*}
The third identity holds by definition of $\xi_f$ and the last identity holds by definition of $T_{0,s}$. 
Thus $\max_{0 < s \leq \|f\|_\infty} s \det T_{0,s} = s_0  \det T_{0} = s_0  \det T_{0, s_0}$. 
Similarly, for all $m \in \mathbb {N}$, there exist  unique $s_m$ and  a unique, modulo $O(n)$,
$T_{m, s_m} \in S_+$  such that 
$$
\xi_{f_m}(s_m)= \max_{0 < s \leq \|f_m\|_\infty} \xi_{f_m}(s) = s_m  \det T_{m, s_m} .
$$
	Since \( f\) is integrable and as  \(\inte(\supp(f))\neq \emptyset\), 
	\[ 0 < \int_{\R^n} f(x) dx=\int_0^{\|f\|_\infty}\vol_n(G_f(s))ds<\infty.\]
Therefore, for  all \(\eps>0\), there exists \(\delta_\varepsilon >0\) such that  for all $0 < \delta < \delta_\varepsilon$, 
 \[ 0 < \int_0^\delta \vol_n(G_f(s))ds<\eps.\]
In particular, for 
$\eps_0\le \min\{\frac{s_0 \det T_{0,s_0}\vol_n(B_2^n)}{2+\|f\|_\infty +\frac{10 \int_{\mathbb{R}^n} f (s) ds}{9 \|f\|_\infty}}, \frac{\|f\|_\infty}{2} \}$ 
there is  $\delta_{\eps_0}$ such that for all 
\( 0<\delta < \min\{ s_0,   \delta_{\varepsilon_0}, \|f\|_2\}\),
	 \[0 < \int_0^\delta \vol_n(G_f(s))ds<\eps_0.\] 
By Lemma \ref{contiofellipsoid}, the map
$$
G_f(s) \to  \max_{\{T\in S_+: \  TB_2^n\subset G_f(s)\}}\det T = \det T_{0,s}
$$
is continuous and  the map 
$$
G_f(s) \to   \vol_n(G_{f}(s))
$$
is also continuous.
Thus, for $0< \eps_1  \leq \min\{\eps_0, \frac{\|f\|_\infty}{10}\}$ given,  there exists $\eta_1= \eta (\eps_1,s)$ such that for all $\eta \leq \eta_1$,
\begin{equation}\label{Gleichung2}
\left| \det T_{0,s} -  \det T_{m,s}\right| < \eps_1,  
\end{equation}
and 
\begin{equation}\label{Gleichung3}
\left|  \vol_n(G_{f_m}(s))  -     \vol_n(G_{f}(s))  \right| < \eps_1,
\end{equation}
whenever $d_H(G_f(s), G_{f_m}(s)) < \eta$.
\par
\noindent
We fix  \( 0<\delta < \min\{ s_0,   \delta_{\varepsilon_0}, \|f\|_\infty - \eps_1\}\).  As  \( f_m\to f\) pointwise, we get, 
similarly  to the proof of Lemma \ref{ptwcontoflevelset}, that for 
all $0 < \alpha $ with $\delta  < \|f\|_\infty - \alpha$,  
 \[ G_{f_m}(s)\to G_f(s)\] in Hausdorff distance,  uniformly for all $s$ with \( \delta \le s\le \|f\|_\infty - \alpha\).	
Thus, in particular for all $s$ with \( \delta \le s\le \|f\|_\infty -\eps_1\), for 	
	\( 0 < \eta < \eta_1\),  there is \(m_1\) such that for all \(m\geq m_1\),  
	\be d_H\left( G_{f_m}(s),G_f(s)\right)<\eta.
	\label{volleveldiff}\ee
By (\ref{Gleichung2}) and (\ref{Gleichung3}) we therefore get that uniformly for all $s$ with \( \delta \le s\le \|f\|_\infty - \eps_1\) and for all \(m\geq m_1\), 
\begin{equation}\label{GL4}
\left| \det T_{0,s} -  \det T_{m,s}\right| < \eps_1,  
\end{equation}
and 
\begin{equation}\label{GL5}
\left|  \vol_n(G_{f_m}(s))  -     \vol_n(G_{f}(s))  \right| < \eps_1.
\end{equation}
\par
\noindent
By Lemma \ref{convergenceinfnorm},  \( f_m\to f\) pointwise implies that $\|f_m\|_\infty \to \|f\|_\infty$, i.e., there is $m_2$ such that 
\begin{equation}\label{fm}
 \|f\|_\infty - \eps_1 < \|f_m\|_\infty  < \|f\|_\infty + \eps_1
\end{equation}  for all $m \geq m_2$.
In addition, by Lemma 3.2 of \cite{AKM2004}, $\int_{\R^n} f_m \ dx \to \int_{\R^n} f\ dx$, i.e. there is $m_3$ such that for all $m \geq m_3$,
\begin{equation}\label{Integral}
\left| \int_{\R^n} f_m(x) \ dx  -  \int_{\R^n} f(x)\ dx \right| < \eps_1.
\end{equation}
\par
\noindent
Let $m_0= \max\{m_1, m_2, m_3\}$. 
Then, on the one hand, it follows with (\ref{GL4}) that  for  all \( m\geq m_0\), all \( \delta <  \min\{ s_0,   \delta_{\varepsilon_0}, \|f\|_2\}\),  all  $s_m$ such that  \( \delta\le s_m\le  \|f\|_\infty - \eps_1\), 
	\begin{eqnarray} \label{limsup1}
	\xi_f(s_0)=s_0\det T_{0,s_0}&\ge& s_m\det T_{0,s_m} 
	\ge s_m(\det T_{m,s_m}-  \eps_1) \nonumber \\
	& \geq&  s_m\det T_{m,s_m} -   \eps_1 (\|f\|_\infty - \alpha) \nonumber \\
	&\geq&\xi_{f_m}(s_m) -   \eps_1 \|f\|_\infty.
	\end{eqnarray}
\par
\noindent
Furthermore, for \( m\geq m_0\) and \( s_m<\delta\), 
	\begin{eqnarray} \label{limsup2}
	 \limsup_{m\to \infty, s_m<\delta} \xi_{f_m}(s_m)&=&\limsup_{m\to \infty, s_m<\delta}s_m\det T_{m,s_m}\le  \frac{ \eps_0  \left(2+ \|f\|_\infty  + \frac{10 \int_{\mathbb{R}^n}  f(s) ds}{9 \|f\|_\infty} \right)}{\vol_n(B^n_2)}  \nonumber\\
	 &\leq& s_0\det T_{0,s_0}=\xi_f(s_0). 
	 \end{eqnarray}
	The last inequality holds by assumption on $\eps_0$. 
We now verify the second last inequality. We have for all $s \leq s_m$ that $G_{f_m}(s_m) \subseteq G_{f_m}(s)$ and therefore  by definition of $T_{m,s_m}$, 
$$
\vol_n(G_{f_m}(s)) \geq \vol_n(G_{f_m}(s_m)) \geq \det T_{m,s_m}  \vol_n(B^n_2) .
$$
Thus, as $s_m < \delta$ and also using (\ref{fm}), 
\begin{eqnarray*}
s_m\det T_{m,s_m} &\le& \frac{1}{ \vol_n(B^n_2)}  \  	\int_0 ^{s_m}  \vol_n(G_{f_m}(s))  \ ds  \leq \frac{1}{ \vol_n(B^n_2)}  \  	\int_0 ^{\delta}  \vol_n(G_{f_m}(s)) \  ds  \\
&=& \frac{1}{ \vol_n(B^n_2)}  \   \left( 	\int_0 ^{ \|f\|_\infty - \eps_1}  \vol_n(G_{f_m}(s))  ds - 	\int_\delta ^{ \|f\|_\infty- \eps_1}  \vol_n(G_{f_m}(s))  ds \right) \\
&\leq& \frac{1}{ \vol_n(B^n_2)}  \   \Bigg( 	\int_0 ^{ \|f_m\|_\infty}  \vol_n(G_{f_m}(s)) \  ds -  	\int_\delta ^{ \|f\|_\infty - \eps_1}  \vol_n(G_{f_m}(s)) \  ds \Bigg) \\
 &=&  \frac{1}{ \vol_n(B^n_2)}  \   \Bigg( \int f_m \ dx  -	\int_\delta ^{ \|f\|_\infty - \eps_1}  \vol_n(G_{f_m}(s)) \  ds \Bigg)\\
&\leq&  \frac{1}{ \vol_n(B^n_2)}  \   \Bigg( \int f\ dx  + \eps_1 -	\int_\delta ^{ \|f\|_\infty - \eps_1}  \vol_n(G_{f_m}(s)) \  ds \Bigg)\\
&=&  \frac{1}{ \vol_n(B^n_2)}  \   \Bigg( \int_0 ^{ \|f\|_\infty}  \vol_n(G_{f}(s)) \  ds  + \eps_1 -	\int_\delta ^{ \|f\|_\infty - \eps_1}  \vol_n(G_{f_m}(s)) \  ds \Bigg).
\end{eqnarray*}	
The last inequality follows by (\ref{Integral}). Now we use (\ref{GL5})
and get that 
	\begin{eqnarray*}
		s_m\det T_{m,s_m} &\le& \frac{1}{ \vol_n(B^n_2)} \Bigg( \int_0 ^{ \|f\|_\infty}  \vol_n(G_{f}(s)) \  ds  + \eps_1  -	\int_\delta ^{ \|f\|_\infty -\eps_1}  \vol_n(G_{f}(s)) \  ds \\
		& +& \eps_1(\|f\|_\infty- \eps_1-\delta) \Bigg) \\
		&\le & \frac{1}{ \vol_n(B^n_2)} \Bigg( \int_0 ^{ \delta}  \vol_n(G_{f}(s)) \  ds  + \int_{\|f\|_\infty -\eps_1} ^{\|f\|_\infty } \vol_n(G_{f}(s)) \  ds \\
		&+& \eps_1 (1+\|f\|_\infty - \delta - \eps_1)  \Bigg) \\
		& \leq& \frac{1}{ \vol_n(B^n_2)} \Bigg( \eps_0 + \eps_1 (1+\|f\|_\infty - \delta - \eps_1)  + \eps_1 \ \vol_n(G_{f}\left(\|f\|_\infty -\eps_1)\right) \Bigg) \\
	& \leq& \frac{1}{ \vol_n(B^n_2)} \Bigg( \eps_0 + \eps_1 \left(1+\|f\|_\infty \right)  + \eps_1  \frac{10 \ \int_{\mathbb{R}^n} f(s) ds}{9 \ \|f\|_\infty} \Bigg) \\
	& =& \frac{1}{ \vol_n(B^n_2)} \Bigg( \eps_0 + \eps_1 \left(1+\|f\|_\infty  + \frac{10 \ \int_{\mathbb{R}^n} f(s) ds}{9 \ \|f\|_\infty} \right)   \Bigg)\\
	& \leq & \frac{\eps_0}{ \vol_n(B^n_2)} \left( 2 + \|f\|_\infty  + \frac{10 \ \int_{\mathbb{R}^n} f(s)  ds}{9 \ \|f\|_\infty} \right).
	\end{eqnarray*}	
	The last inequality follows by choice of $\eps_1$. The second last inequality above follows as for all $\|f\|_\infty - \eps_1 \leq s \leq \|f\|_\infty$, we have that $ \vol_n(G_{f}(s)) \leq  \vol_n(G_{f}(\|f\|_\infty -\eps_1))$ and as 
	$$ 
	\vol_n(G_{f}(\|f\|_\infty -\eps_1) \leq  \frac{\int_{\mathbb{R}^n} f(s) ds}{\|f\|_\infty -\eps_1}  \leq  \frac{10 \ \int_{\mathbb{R}^n} f(s) ds}{9 \ \|f\|_\infty }, 
	$$
	 by choice of $\eps_1$. Now we use how $\eps_0$ was chosen and get that  for all  $s_m < \delta$, 
\begin{eqnarray*}
		s_m\det T_{m,s_m} &\le&	  s_0\det T_{0,s_0}=\xi_f(s_0). 
\end{eqnarray*}	
\par
\noindent
It remains to check when $ \|f\|_\infty - \eps_1 \leq s_m \leq  \|f\|_\infty +\eps_1$.
\begin{eqnarray*}
\xi_{f_m}(s_m) &=& s_m \det T_{m, s_m}     \\
                      &\le & (\|f\|_\infty+\eps_1)(\det T_{m,\|f\|_\infty-\eps_1})\\
                      &\le & (\|f\|_\infty+\eps_1)(\det T_{0,\|f\|_\infty-\eps_1}+\eps_1)\\
                      &=& (\|f\|_\infty-\eps_1+2\eps_1)(\det T_{0,\|f\|_\infty-\eps_1}+\eps_1)\\
                      &=& (\|f\|_\infty-\eps_1)\det T_{0,\|f\|_\infty-\eps_1}+2\eps_1\det T_{0,\|f\|_\infty-\eps_1}+\eps_1(\|f\|_\infty-\eps_1)+2\eps_1 ^2\\
                      &\le &  \xi_f(s_0) +2\eps_1\det T_{0,\|f\|_\infty-\eps_1}+\eps_1 \|f\|_\infty+\eps_1^2\\
                      &\le & \xi_f(s_0) + 2\eps_0  \left(\det T_{0,\frac{\|f\|_\infty}{2}} + \frac{3}{4} \|f\|_\infty\right).
\end{eqnarray*}
In the first inequality, \( s_m< \|f\|_\infty+\eps_1\) by assumption, and  \(\det T_{m, s_m}\le  \det T_{m,\|f\|_\infty-\eps_1}\) since \( G_{f_m}(s_m)\subset G_{f_m}(\|f\|_\infty-\eps_1)\). In the second inequality, we apply (\ref{GL4}). In the last inequality, we use \( \eps_1<\eps_0\) as assumed, and  we  also use the assumption on \( \eps_0\). 
Therefore, we have for all \( m\geq m_0\) and \( \|f\|_\infty-\eps_1<s_m<\|f\|_\infty+\eps_1\),
\begin{equation}  \label{limsupA3}
\xi_f(s_0)\ge \xi_{f_m}(s_m) -2\eps_0 \left(\det T_{0,\frac{\|f\|_\infty}{2}} + \frac{3}{4} \|f\|_\infty\right)
\end{equation}
\par
\noindent	
	It now follows  from (\ref{limsup1}),  (\ref{limsup2})  and  (\ref{limsupA3}) that   
	\begin{equation}\label{limsup3}
	 \xi_f(s_0)\ge \limsup_m \xi_{f_m}(s_m).	 
	 \end{equation}
\par
\noindent	
    On the other hand, as $\delta \leq s_0$,  for $\eps_1$ given,  it follows from (\ref{GL4})  that for all \(m \geq m_0\), 
    	\[\det T_{0,s_0} \le \det T_{m,s_0} +  \eps_1 .\]
	Therefore, for all \( m \geq m_0\),
\begin{eqnarray*}
	s_0 \det T_{0,s_0} \le s_0 \det T_{m,s_0}  + s_0  \eps_1 \le s_m\det T_{m,s_m} + s_0  \eps_1.
\end{eqnarray*}		
	The last inequality holds as  \( s_m\det T_{m,s_m}=\max_{\{T\in S_+: TB_2^n\subset G_{f_m}(s)\}} s\det T\). 	
	Consequently, for all $m \geq m_0$, 
		\[\xi_f(s_0) =  s_0\det T_{0,s_0}\le s_m\det T_{m,s_m} + s_0  \eps_1 = \xi_{f_m}(s_m) + s_0  \eps_1, \]
		and hence 
\begin{equation}\label{limsup4}		
		\xi_f(s_0)\le \liminf_m \xi_{f_m}(s_m).
\end{equation}			
\par
\noindent		
Altogether, by (\ref{limsup3}) and (\ref{limsup4}), 
	\[\limsup_m\xi_{f_m}(s_m)\le \xi_f(s_0)\le \liminf_m\xi_{f_m}(s_m),\] 
	and thus 
	\begin{equation}\label{Limes}
	 \lim_m\xi_{f_m}(s_m) = \xi_f(s_0).
	 \end{equation}
	   By (\ref{xi-Gleichung}), this is equivalent to 		
		\[\lim_m \max_{0<s\le \|f_m\|_\infty}\xi_{f_m}(s) =  \max_{0<s\le \|f\|_\infty}\xi_f(s).\]
	
\end{proof}
\vskip 3mm
\noindent
In fact, (\ref{Limes}) together with  (\ref{Gleichung}) says that if \(\{f_m\},f\) are integrable,  log-concave functions and if  \( f_m\to f\) pointwise,  then 
\begin{eqnarray*}
 && \hskip -15mm \max \{ e^{-t} \ \det T : \   T\in S_+,  \  t\in \R, \  e^{-t}\1_{TB_2^n} \  \le \  f_m\}  \to \\
 && \hskip 20mm \max \{ e^{-t} \ \det T : \   T\in S_+,  \  t\in \R, \  e^{-t}\1_{TB_2^n} \  \le \  f\} .
\end{eqnarray*}
\vskip 2mm
\noindent
Thus we immediately get the following corollary.
\vskip 2mm
\noindent
\begin{corollary} \label{Korollar}
Let 	\(\{b_m\},b\) in $\mathbb{R}^n$  be such that \(b_m \to b\). Then
\begin{eqnarray*}
 && \hskip -15mm \max \{ e^{-t} \ \det T : \   T\in S_+,  \  t\in \R, \  \|Tx\|_2 -t \leq \psi(x-b_m)\}  \to \\
 && \hskip 20mm \max \{ e^{-t} \ \det T : \   T\in S_+,  \  t\in \R, \     \|Tx\|_2 -t \leq \psi(x-b)\} .
\end{eqnarray*}
\end{corollary}
\vskip 4mm

\begin{proof}
We have that
$$
\forall x\in\mathbb R^{n}:
\hskip 3mm
\|Tx\|_2-t\leq\psi(x-b_{m})
$$
is equivalent to
$$
\forall y\in\mathbb R^{n}:
\hskip 3mm
\|T(y+b_{m})\|_2-t\leq\psi(y).
$$
We put
$$
B_{m}
=\{(T,t):\hskip 1mm\|T(y+b_{m})\|_2-t\leq\psi(y)\hskip 1mm \forall y\in\mathbb R^{n}\}
$$
and
$$
B_{}
=\{(T,t):\hskip 1mm\|T(y+b_{})\|_2-t\leq\psi(y)\hskip 1mm \forall y\in\mathbb R^{n} \}.
$$
Then
\begin{eqnarray*}
\max\{e^{-t}\det T:(T,t)\in B_{m}\}
=e^{\|T(b-b_{m})\|_2}\max\{e^{-t-\|T(b-b_{m})\|_2}\det(T):(T,t)\in B_{m}\}.
\end{eqnarray*}
Since
$$
\|T(y+b_{m})\|_2\geq\|T(y+b)\|_2-\|T(b-b_{m})\|_2
$$
we get
\begin{eqnarray*}
\max\{e^{-t}\det T:(T,t)\in B_{m}\}
\leq e^{\|T(b-b_{m})\|_2}\max\{e^{-s}\det(T):(T,s)\in B\}.
\end{eqnarray*}
It follows that
\begin{eqnarray*}
\limsup_{m\to\infty}\max\{e^{-t}\det T:(T,t)\in B_{m}\}
\leq \max\{e^{-s}\det(T):(T,s)\in B\}.
\end{eqnarray*}
Now we interchange the roles of $b$ and $b_{m}$ and get
\begin{eqnarray*}
\max\{e^{-s}\det T:(T,s)\in B\}
\leq \liminf_{m\to\infty}\max\{e^{-t}\det(T):(T,t)\in B_{m}\}.
\end{eqnarray*}	
\end{proof}

\vskip 2mm
\noindent
The proof of  Theorem \ref{lownerfcn} is next.
\begin{proof}
Let  \( f=e^{-\psi}\) be an integrable log-concave function with positive integral. We put 
$$ 
I_f:=\min\left\lbrace \int_{\R^n}e^{-\|Ax\|_2+t}dx: A\in \mathcal{A},t\in \R, \|Ax\|_2-t\le \psi(x) \right\rbrace 
$$
and 
$$
I_f(b):= \min\left\lbrace \int_{\R^n}e^{-\|Ax\|_2+t}dx: A\in \mathcal{A}(b),t\in \R, \|Ax\|_2-t\le \psi(x) \right\rbrace. 
$$
It follows from (\ref{equivalent2}) that \( I_f=\min_{b\in \R^n}I_f(b)\). 
\par
\noindent
By the reduction arguments in Section \ref{Redux}, 
\begin{eqnarray*}
I_f(b) &=& \min\left\lbrace \int_{\R^n}e^{-\|Ax\|_2+t}dx: A\in \mathcal{A}(b),t\in \R, \|Ax\|_2-t\le \psi(x) \right\rbrace\\
	    &=& n!\vol_n(B_2^n) \   \min\left\{ \frac{e^t}{\det T}: T\in S_+, t\in \R, \|Tx\|_2-t\le \psi(x-b)\right\}\\
	    &=&n!\vol_n(B_2^n) \   \left\lbrace  \max\lbrace  e^{-t}\det T:  T\in S_+, t\in \R, \|Tx\|_2-t\le \psi(x-b)     \rbrace   \right\rbrace ^{-1}.
	  \end{eqnarray*} 
Corollary \ref{Korollar}  implies  that \( I_f(b)\) is continuous in \(b\). To see that the minimum \( I_f\) exists, it suffices to show that the minimum is achieved on a compact set. 
\par
\noindent
Let \(0<d_0 < \|f\|_\infty\) be such that \( G_f(d_0)\) has positive volume.  Let  \( b_0\in G_f(d_0)\). Clearly, \( I_f\le I_f(b_0)\). Let \( r=\frac{I_f(b_0)}{d_0}-\vol_n(G_f(d_0))\). 
Then \(r>0\) since 
\[d_0\vol_n(G_f(d_0))< \int_{\R^n} f (x) dx \le I_f(b_0).\]
The last inequality holds as 
\begin{eqnarray*}
I_f(b_0) &=& \min\left\lbrace \int_{\R^n}e^{-\|Ax\|_2+t}dx: A\in \mathcal{A}(b_0),t\in \R, \|Ax\|_2-t\le \psi(x) \right\rbrace  \\
&=&  \min\left\lbrace \int_{\R^n}e^{-\|Ax\|_2+t}dx: A\in \mathcal{A}(b_0),t\in \R,  e^{-\|Ax\|_2 + t } \ge f(x) \right\rbrace .
\end{eqnarray*}
To finish the existence argument,   we need the notion of illumination body of a convex body $K$. This notion was introduced in \cite{Werner1994} as follows. Let $\delta >0$ be given.
The  illumination body $K^\delta$ of  $K$  is  
\[K^{\delta}=\{x\in\mathbb{R}^n: \vol_n( \mathrm{conv}[K,x])\leq \delta + \vol_n(K)\}\quad.\]
The illumination body is always convex, \cite{Werner1994}.  See, e.g.,  \cite{MordhorstWerner1, MordhorstWerner2} for recent developments.
\par
\noindent
 Let now  \( G^r= [G_f(d_0)]^r\) be the illumination body of \( G_f(d_0)\).  We will show that for \( b\notin G^r\), \( I_f(b)> I_f(b_0)\). 
 Suppose \( b\notin G^r\) and  let \(  A_0\in \mathcal{A}(b),t_0\in \R \) achieve \(I_f(b)\). Let \(h(x)=e^{-\|A_0x\|_2+t_0}\). Since $G_{f}(d_{0})\subseteq G_{h}(d_{0})$ there exists \(z\in bd(G^r)\cap G_h(d_0)\) such that 
\[\conv [z, G_f(d_0)]\subsetneq G_h(d_0).\]
It follows that 
\begin{eqnarray*}
I_f(b)&=&\int_{\R^n} h(x) dx > d_0 \cdot  \  \vol_n(G_h(d_0)) > d_0 \   \vol_n(\conv [z, G_f(d_0)]) \\
& = &
 d_0\cdot (r+\vol_n(G_f(d_0)))=I_f(b_0) .
\end{eqnarray*}
So for the minimization problem, we need only consider \( b\in G^r\) where \(G^r=[G(d_0)]^r\) is a compact set of \(\R^n\). The continuity of \(I_f(b)\) gives the existence of a minimizer.
\par
\noindent
Next we address the uniqueness. 
Recall that \( I_f=\min_b I_f(b) \) and  Proposition \ref{reducedprob} guarantees that for each \(b\in \R^n\) there is a unique, modulo $O(n)$,  minimizer. Hence it suffices to show that there is a unique \(b_0\) such that \(  I_f=\min_b I_f(b)=I_f(b_0)\).
\par
\noindent
We prove by way of contradiction. Suppose that there are \( b_1, b_2\) such that \( I_f=I_f(b_1)=I_f(b_2)\) and \( b_1\neq b_2\). Let the two minimizers corresponding to \(b_1\) 
and \( b_2 \) be \( (T_1, t_1) \in S_+ \times \mathbb{R} \) and \( (T_2, t_2) \in S_+ \times \mathbb{R}\),  respectively. $T_1$ and $T_2$ are unique up to an orthogonal transformation.
Then for all $x\in\mathbb R^{n}$
\begin{equation*}  \|T_1(x+b_1)\|_2-t_1\le \psi(x),  \hskip 5mm
 \|T_2(x+b_2)\|_2-t_2\le \psi(x) 
 \end{equation*} 
and  
$$\frac{e^{t_1}}{\det T_1 }=\frac{e^{t_2}}{\det T_2 }, \label{twominimizers}
$$
or, equivalently, taking logarithm on both sides, 
\be \label{Gleich}
t_1-\log \det T_1=t_2-\log \det T_2.
\ee
We distinguish two cases.
\par
\noindent
Case 1.  \( T_1\neq T_2\). Then we consider the function 
\[e^{-\left\|\frac{T_1+T_2}{2}x+\frac{T_1b_1+T_2b_2}{2}\right\|_2+\frac{t_1+t_2}{2}}.\]
Observe that 
\begin{eqnarray*}
	&&\left\|\frac{T_1+T_2}{2}x+\frac{T_1b_1+T_2b_2}{2}\right\|_2-\frac{t_1+t_2}{2}
	= \left\| \frac{1}{2} T_1(x+b_1)+\frac{1}{2} T_2(x+b_2)\right\|_2-\frac{t_1+t_2}{2}\\
	&& \hskip 10mm \le \frac{1}{2}\left( \|T_1(x+b_1)\|_2-t_1\right)+\frac{1}{2}\left( \|T_2(x+b_2)\|_2-t_2\right)
	\le \psi(x).
\end{eqnarray*}
But 
\[ \int_{\R^n} e^{-\left\|\frac{T_1+T_2}{2}x+\frac{T_1b_1+T_2b_2}{2}\right\|_2+\frac{t_1+t_2}{2}} dx=n!\vol_n(B_2^n) \frac{e^{\frac{t_1+t_2}{2}}}{\det(\frac{T_1+T_2}{2})}.\]
And by the Minkowski determinant inequality, 
$$
\left(\det\left(\frac{T_1+T_2}{2}\right) \right)^\frac{1}{n} \geq \frac{1}{2} \left(  \det (T_1) ^\frac{1}{n} + \det (T_2) ^\frac{1}{n}\right),
$$
from which it follows by concavity of the logarithm that
$$
\log\det\left(\frac{T_1+T_2}{2}\right) > \frac{1}{2} \left( \log\det T_1 + \log\det T_2 \right).
$$
The inequality is strict because the function \( T\to -\log\det T\) is strictly convex on the set of positive definite matrices.
Hence 
\begin{eqnarray*}
	\log \frac{e^{\frac{t_1+t_2}{2}}}{\det(\frac{T_1+T_2}{2})} &=& \frac{t_1+t_2}{2}-\log\det\left(\frac{T_1+T_2}{2}\right)\\
	&<& \frac{t_1}{2}-\frac{1}{2}\log \det T_1+\frac{t_2}{2}-\frac{1}{2}\log \det T_2\\
	&=& t_1-\log \det T_1=t_2-\log \det T_2.
\end{eqnarray*}
It follows that 
\[ \frac{e^{\frac{t_1+t_2}{2}}}{\det(\frac{T_1+T_2}{2})} < \frac{e^{t_1}}{\det T_1}, \]
which contradicts the fact that the latter is the minimum.
\par
\noindent
Case 2.  \( T_1=T_2\), modulo $O(n)$. It follows from (\ref{Gleich}) that \( t_1=t_2\). 
We show $b_{1}=b_{2}$.
We put 
\[ f_1(x)=e^{-\|T_1(x+b_1)\|_2+t_1}\]
and
\[f_2(x)=e^{-\|T_2(x+b_2)\|_2+t_2}=e^{-\|T_1(x+b_2)\|_2+t_1}.\]
We consider  super-level sets. For \( 0<s<e^{t_1}\), 
\begin {equation*} 
G_{f_1}(s)=-b_1+(t_1-\log s) T_1^{-1}B_2^n 
\end{equation*}
and
\begin {equation*}  
G_{f_2}(s)=-b_2+(t_1-\log s) T_1^{-1}B_2^n.
\end {equation*} 
For \( 0<s<\|f\|_\infty\), one has by definition of $f_1$ and $f_2$ that 
\[ G_f(s)\subset G_{f_1}(s)\cap G_{f_2}(s).\]
Now we claim that 
\begin {equation*} 
G_{f_1}(s)\cap G_{f_2}(s)\subset  -\frac{b_1+b_2}{2}+(t_1-\log s) T_1^{-1}B_2^n .
\end {equation*} 
If \( G_{f_1}(s)\cap G_{f_2}(s)=\emptyset\),  this inclusion is trivially true.  If not, let \( x\in G_{f_1}(s)\cap G_{f_2}(s)\). Then there exist \( u,v\in B_2^n\) such that 
\be
x= -b_1+(t_1-\log s) T_1^{-1}u=-b_2+(t_1-\log s) T_1^{-1}v. \label{x-darstell}
\ee
Thus 
\be
	x=\frac{x+x}{2}=-\frac{b_1+b_2}{2}+(t_1-\log s) T_1^{-1}\left( \frac{u+v}{2}\right) \label{intersectionpt}
\ee
Since \( \|(u+v)/2\|\le \|u\|/2+\|v\|/2\le 1\), 
\[
x=-\frac{b_1+b_2}{2}+(t_1-\log s) T_1^{-1}\left( \frac{u+v}{2}\right) \in -\frac{b_1+b_2}{2}+(t_1-\log s) T_1^{-1}B_2^n.
\]
In the following we show that there is  \( \widetilde{T_1}\) with \( \det(\widetilde T_1)>\det(T_1)\) satisfying 
\[
G_{f_1}(s)\cap G_{f_2}(s)\subset  -\frac{b_1+b_2}{2}+(t_1-\log s) \widetilde{T_1}^{-1}B_2^n\subset  -\frac{b_1+b_2}{2}+(t_1-\log s) T_1^{-1}B_2^n .
\]
Both,  \( G_{f_1}(s)\cap G_{f_2}(s)\) and \( -\frac{b_1+b_2}{2}+(t_1-\log s) T_1^{-1}B_2^n\),  are closed sets and centrally symmetric with respect to the same center 
\( -\frac{b_1+b_2}{2}\).  
\par
\noindent
Next we observe that  \( G_{f_1}(s)\cap G_{f_2}(s)\) does not intersect the boundary of the ellipsoid \( -\frac{b_1+b_2}{2}+(t_1-\log s) T_1^{-1}B_2^n\). Indeed, If \( x\in G_{f_1}(s)\cap G_{f_2}(s)\) as represented in (\ref{intersectionpt}) is on the boundary of 
\( -\frac{b_1+b_2}{2}+(t_1-\log s) T_1^{-1}B_2^n\), it follows that \( u=v\in S^{n-1}\). Hence by (\ref{x-darstell}),  \( b_1=b_2\),   a contradiction.  
\par
\noindent
Therefore \( G_{f_1}(s)\cap G_{f_2}(s)\) is a convex body  such that 
\[\left( G_{f_1}(s)\cap G_{f_2}(s)\right)  \cap \overline{\left( -\frac{b_1+b_2}{2}+(t_1-\log s) T_1^{-1}B_2^n\right) ^c}=\emptyset,\]
and thus 
\[ dist\left( G_{f_1}(s)\cap G_{f_2}(s), \overline{\left( -\frac{b_1+b_2}{2}+(t_1-\log s) T_1^{-1}B_2^n\right) ^c}\right) >0,\]
where \( dist(A,B)=\inf\{ \|x-y\|_2, x\in A, y\in B\}\).  Hence  we may shrink the ellipsoid 
\( -\frac{b_1+b_2}{2}+(t_1-\log s) T_1^{-1}B_2^n\) with respect to the center \( -(b_1+b_2)/2\) homothetically to get 
a new ellipsoid \( -\frac{b_1+b_2}{2}+(t_1-\log s) \widetilde {T_1}^{-1}B_2^n\)
such that still
\[G_{f_1}(s)\cap G_{f_2}(s)\subset  -\frac{b_1+b_2}{2}+(t_1-\log s) \widetilde{T_1}^{-1}B_2^n\]
and such that 
\( -\frac{b_1+b_2}{2}+(t_1-\log s) \widetilde {T_1}^{-1}B_2^n\) intersects the boundary of  \( G_{f_1}(s)\cap G_{f_2}(s)\).
Given such  a \( \widetilde {T_1}^{-1}\), it follows from 
\[G_f(s)\subset G_{f_1}(s)\cap G_{f_2}(s)\subset  -\frac{b_1+b_2}{2}+(t_1-\log s) \widetilde{T_1}^{-1}B_2^n\]
that 
\[ \left\| \widetilde{T_1}\left(x+\frac{b_1+b_2}{2}\right)\right\|_2-t_1
\le \psi(x).\]
However 
\[ \int_{\R^n}  e^{-\left\| \widetilde{T_1}\left(x+\frac{b_1+b_2}{2}\right)\right\|_2+t_1} dx=n!\vol(B_2^n) \frac{e^{t_1}}{\det \widetilde{T_1}}< n!\vol(B_2^n) \frac{e^{t_1}}{\det T_1}, 
\]
which is a contradiction. 
\par
\noindent
Consequently, we have proved that \( b_1=b_2\). 
\end{proof}

\section{John function and duality}

\subsection{The John function of Alonso-Guti{\'e}rrez,   Merino, Jim{\'e}nez, 
	and Villa}

A notion of a  John ellipsoid function has already been  introduced in \cite{Alonso-Gutiérrez2017}, Theorem 1.1.
We first recall the definition from this work.

\begin{theorem}[\cite{Alonso-Gutiérrez2017}] \label{johnfunction}
	Let \(f:\R^n\to \R\) be an integrable log-concave function. There exists  a unique   solution \( (s_0, A_0 )\in \mathbb{R} \times \mathcal{A}\)  to the maximization problem
\begin{equation} \label{JE}
\max\{ s |\det A| :  s \le \|f\|_\infty, A\in \mathcal{A} \} \hskip 3mm  \text{subject to} \hskip 3mm   s \  \1_{A B_2^n}\le f.
\end{equation}
$A_0$  is unique  up to  right orthogonal transformations. 
Then $s_0  \  \1_{A_0 B_2^n}$ is called the John ellipsoid of $f$, $J(f)=s_0  \  \1_{A_0 B_2^n}$.
\end{theorem}
\par
\noindent
Note that, as for the L\"owner function, $J(f)=s_0  \  \1_{A_0 B_2^n}$ up to an orthogonal transformation.
\par
\noindent
We show that Theorem \ref{johnfunction} can be obtained from  Proposition \ref{equivreducedprob} and  Lemma \ref{continuityI(b)}. 
However, it seems that  Theorem \ref{lownerfcn} cannot be obtained immediately from Theorem \ref{johnfunction}
as the optimization in (\ref{JE}) is over all affine maps, i.e., translation is allowed under the constraint that $s\1_{AB_2^n}\le f$. 
To see how Theorem \ref{johnfunction} follows from  Proposition \ref{equivreducedprob} and Lemma \ref{continuityI(b)}, 
we first rewrite (\ref{JE}) in  Theorem \ref{johnfunction}. Let \(A=T-b\), 
\begin{eqnarray*} 
	s_0\det A_0 &=& \max \{ s \det A:  s\le \|f\|_\infty, A\in \mathcal{A},          s \1_{AB_2^n}\le f \} \nonumber \\
	&=& \max \{ s \det T: s\le \|f\|_\infty, T\in S_+,  b\in \R^n,         s\1_{TB_2^n-b}(x)\le f(x) \   \forall x \in \mathbb{R}^n \}  \nonumber \\
	&=& \max \{ s\det T: s\le \|f\|_\infty, T\in S_+,  b\in \R^n,          s\1_{TB_2^n}(x+b)\le f(x) \  \forall x \in \mathbb{R}^n \} \nonumber \\
	&=& \max \{ s\det T: s\le \|f\|_\infty, T\in S_+,  b\in \R^n,          s\1_{TB_2^n}(x)\le f(x-b)\  \forall x \in \mathbb{R}^n \} \nonumber \\
	&=& \max_{b \in \R^n}\ \  \max\{ s \det T: s\le \|f\|_\infty, T\in S_+,         s\1_{TB_2^n}(x)\le f(x-b) \   \forall x \in \mathbb{R}^n\} 
\end{eqnarray*}
If we put \( J=s_0\det A_0\) and 
\[
J_f(b)=\max\{ s \det T:  s\le \|f\|_\infty, T\in S_+,         s \1_{TB_2^n}(x)\le f(x-b)\   \forall x \in \mathbb{R}^n  \}, 
\]
then \( J=\max_{b\in \R^n} J_f(b)\). Note also that \( J_f(b)\) is continuous in \(b\) by   Lemma \ref{continuityI(b)}.
\vskip 2mm
\noindent
We show  now that the existence of the John function follows from  Proposition \ref{equivreducedprob} and Lemma \ref{continuityI(b)}.
\par 
\noindent
{\bf Existence of the John function in Theorem \ref{johnfunction}.} 
Recall that 
existence and uniqueness of \(J_f(b)\) are proved in  Proposition \ref{equivreducedprob} . Choose \(b'\in \R^n\) such that \( J_f(b')>0\). Now let \( \eps=J_f(b')\). Since \(f\) is integrable, there exists \(\delta(\eps)\) such that 
\[ \int_0^{\delta(\eps)} \vol_n(G_f(s))ds<\eps.\] 
Then for \( b\notin G_f(\delta(\eps))\), \( J_f(b)<\eps\). In fact, \[ J_f(b)\le \int_0^{\delta(\eps)} \vol_n(G_f(s))ds<\eps.\] 
Hence  \[ \max_{b\in \R^n} J_f(b)= \max_{b\in G_f(\delta(\eps))} J_f(b).\]
Since \( G_f(\delta(\eps))\) is compact  and \( J_f(b)\) is continuous in \(b\) by    Lemma \ref{continuityI(b)}, \( \max_{b\in \R^n} J_f(b)= \max_{b\in G_f(\delta(\eps))} J_f(b)\) exists. 
\par 
\noindent
We include the uniqueness argument for the reader's convenience.
\par 
\noindent
{\bf Uniqueness of the John function in Theorem \ref{johnfunction}.} Suppose that \( \max_{b\in \R^n} J_f(b)= J_f(b_1)=J_f(b_2)\) for some \( b_1\neq b_2\). If \(b_1=b_2\), then the solution is unique  modulo $O(n)$, by  Proposition \ref{equivreducedprob} . Suppose that \( t_1,t_2, T_1, T_2\) are maximizers satisfying 
\[ J_f(b_1)= t_1\det T_1\ \ \ \text{and} \ \ \ J_f(b_2)=t_2\det T_2.\]
\[ f(T_1v+b_1)\ge t_1\ \ \ \text{and} \ \ \  f(T_2v+b_2)\ge t_2, \ \forall\ v\in B_2^n.\]
Thus we have 
$$ \log t_1+ \log\det T_1= \log t_2+\log\det T_2\label{logmaximum} .$$
We may furthermore assume that \( t_1\neq t_2\). Indeed, observe that \( T_1B_2^n+b_1\) is the John ellipsoid of \(G_f(t_1)\)  and \( T_2B_2^n+b_2\) is the John ellipsoid of \( G_f(t_2)\). If \( t_1=t_2\), then $G_f(t_1)= G_f(t_2)$ and  by the uniqueness of John ellipsoid of a convex body \cite{GardnerBook2006,john1948,SchneiderBook}, \( T_1=T_2\). Hence without loss of generality, we assume \( t_1<t_2\). 
\par 
\noindent
Now we consider the function \[ \sqrt{t_1t_2}\1_{\frac{T_1+T_2}{2}B_2^n+\frac{b_1+b_2}{2}}.\]
We first show that 
$$ \sqrt{t_1t_2}\1_{\frac{T_1+T_2}{2}B_2^n+\frac{b_1+b_2}{2}}\le f.$$
In fact, by the concavity of \( \log f\), we have for any \( u\in B_2^n\), 
\begin{eqnarray*}
	\log f \left( \frac{T_1+T_2}{2}u+\frac{b_1+b_2}{2}\right)
	&\ge& \frac{1}{2}\log f \left(T_1u+b_1 \right)+\frac{1}{2}\log f \left(T_2u+b_2 \right)  \\
	&\ge & \frac{1}{2}\log t_1+\frac{1}{2}\log t_2=\log \sqrt{t_1t_2}.
\end{eqnarray*}
However, \( \sqrt{t_1t_2}\det (\frac{T_1+T_2}{2})>J_f(b_1)\). Indeed, using again the strict concavity of the function \( T\to \log\det T \) on positive definite operators we have 
\begin{eqnarray*}
	\log\left( \sqrt{t_1t_2}\det \left( \frac{T_1+T_2}{2}\right) \right) 
	&=& \frac{1}{2}\log t_1+\frac{1}{2}\log t_2+\log\det  \left( \frac{T_1+T_2}{2}\right) \\
	&>&\frac{1}{2}\log t_1+\frac{1}{2}\log t_2+\frac{1}{2}\log\det  T_1+\frac{1}{2}\log\det  T_2\\
	&=& \frac{1}{2}\left( \log t_1+\log\det  T_1\right) +\frac{1}{2}\left( \log t_2+\log\det   T_2\right) \\
	&=&  \log t_1+\log\det  T_1= \log(J_f(b_1)), 
\end{eqnarray*}
which is a contradiction to the assumption that \( \max_{b\in \R^n} J_f(b)= J_f(b_1)\). 
Consequently, \(b_1=b_2\). 

\subsection{Duality}
Let  $K$ be  a convex body in $\mathbb{R}^n$ such that $0$ is the center of the L\"owner ellipsoid $L(K)$. Then it holds that $(L(K))^\circ = J(K^\circ)$, where $J(K^\circ)$ is the John ellipsoid of $K^\circ$.
This duality relation carries over when we consider the convex bodies in the functional setting.

\begin{proposition}
Let $K$ be a convex body in  $\mathbb{R}^n$. Assume,  without loss of generality, that $0$ is the center of the L\"owner ellipsoid $L(K)$ of $K$.
Then
$$
(L(\1_K))^\circ = J((\1_K)^\circ).
$$
\end{proposition}
\begin{proof}
It was shown in (\ref{LownerK}), that 
$ L (\1_{K})(x) =  e^{-n \left(\|T_{L(K)}^{-1} x\|_2 -1\right)}$. Then
\begin{eqnarray*}
\mathcal{L}\left(n \left(\|T_{L(K)}^{-1} x\|_2 -1\right)\right) (y) &=&n + \sup_{x \in \mathbb{R}^n} \langle x, y \rangle - n \|T_{L(K)}^{-1} x \|_2 \\
&=&n + \sup_{z \in \mathbb{R}^n} \langle z, T_{L(K)}^{t}y \rangle - n \|z\|_2 \\
&=&n + \sup_{z \in \mathbb{R}^n}  \|z\|_2 \left( \|T_{L(K)}^{t}y \|_2 - n \right)\\
&=& n +\begin{cases} 
		\infty & y\notin n (T_{L(K)}^{t})^{-1} B_2^n \\
		0 & y\in n (T_{L(K)}^{t})^{-1} B_2^n .
	\end{cases}
\end{eqnarray*}
Hence, 
$$
(L (\1_{K}))^\circ = e^{-n} \1_{n (T_{L(K)}^{t})^{-1} B_2^n}  = e^{-n} \1_{n J(K^\circ)}.
$$
The last identity holds as $L(K) = T_{LK} B^n_2$, and thus  $J(K^\circ)= (L(K))^\circ = (T_{LK} B^n_2)^\circ=(T_{L(K)}^{t})^{-1} B_2^n$.
Now we compute  \( (\1_K)^\circ =(e^{-I_K} )^\circ\),  where 
\[ I_K(x)=\begin{cases} 
0 & x\in K \\
\infty & x\notin K.
\end{cases}
	\]
The Legendre transform of \( I_K\) is 
\begin{eqnarray*}
	\mathcal{L}(I_K)(y) &=& \sup_{x\in \R^n} \langle x,y\rangle -I_K(x)
	                    = \sup_{x\in K} \langle x,y\rangle
	                    = h_{K}(y),
\end{eqnarray*}
where $h_K$ is the support function of $K$. \( K^\circ\) is a convex body since \( 0\) is contained in the interior of \(K\). Thus,  \( (\1_K)^\circ (y)=e^{-h_K(y)}\).
Next we compute the John function \( J((\1_K)^\circ)\) of \( (\1_K)^\circ\). For \( 0<s\le 1\),
\[ e^{-h_{K}(y)}\ge s \Leftrightarrow h_{K}(y) \le -\log s \Leftrightarrow y\in (-\log s )K^\circ. \]
So the super-level set of \((\1_K)^\circ\) at \(s\) is \(G_{(\1_K)^\circ}(s)=(-\log s)K^\circ\). Moreover, 
$$
J(-\log s\  K^\circ)=  -\log s \  J(K^\circ) =  -\log s  (L(K) )^\circ
$$
and $\max_{s} s (-\log s)^n$ is reached at $s=e^{-n}$. Thus $J((\1_K)^\circ)= e^{-n} \1_{n J(K^\circ)}$.

\end{proof}

\vskip 3mm
\noindent
In a functional  context, we view  as \textit{ellipsoidal functions} or, {\em ellipsoids} in short,  functions of the form 
\[ t \1_\mathcal{E} \text{ and }  \exp(-\|Tx+a\|_2+t), \]
 where \( \mathcal{E}\) is an ellipsoid in \(\R^n\) and \(t\in \R, a\in \R^n, T\in S_+\). 
\par
\noindent
We want to establish a duality relation between the ellipsoidal functions, similar to the one that holds for convex bodies. 
As in the case of convex bodies, we can only expect such a duality relation if we take polarity with respect to the proper 
point. Indeed, 
let \(f=e^{-\psi}\) be a log-concave function.  Let $L(f) (x)= e^{- \|T_0x +a_0\| +t_0}$  be the L\"owner function of $f$. 
Let $b \in \mathbb{R}^n$. Then
\begin{eqnarray*} 
\mathcal{L}_b\left(\|T_0x +a_0\|_2 +t_0\right) (y) &=& t_0 +  \sup_{x \in \mathbb{R}^n} \langle x-b, y -b \rangle -  \|T_0x +a_0\|_2 \nonumber \\
&=& t_0 +  \sup_{z \in \mathbb{R}^n} \langle T_0^{-1} (z-a_0)-b, y -b \rangle -  \|z\|_2 \nonumber \\
&=& t_0 -   \langle b, y -b \rangle - \langle T_0^{-1} (a_0), y -b \rangle + \sup_{z \in \mathbb{R}^n} \langle z, T_0^{-1}(y -b) \rangle -  \|z\|_2  \nonumber\\
&=& t_0 -   \langle b, y -b \rangle - \langle T_0^{-1} (a_0), y -b \rangle + \sup_{z \in \mathbb{R}^n}  \|z\|_2 \left(\|T_0^{-1}(y -b)\|_2 -1\right) \nonumber \\
&=& t_0 -   \langle b, y -b \rangle - \langle T_0^{-1} (a_0), y -b \rangle +\begin{cases} 
		\infty & y\notin T_0  B_2^n + b \\
		0 & y\in T_0 B_2^n + b
	\end{cases}
\end{eqnarray*}
and $(L(f))^b = e^{- \mathcal{L}_{b}( \|T_0x +a_0\|_2 +t_0)}$ is again an ellipsoidal function if and only if $b = b_0= - T_0^{-1} a_0$.  In this case
$$
(L(f))^{-b_0} = e^{- \mathcal{L}_{-b_0}( \|T_0x +a_0\|_2 +t_0)} = e^{-t_0} \   \1_{T_0 B_2^n -b_0}.
$$
\vskip 2mm
\noindent
For log-concave functions $f=e^{-\psi}$ that are even, i.e.,  \( \psi(x)=\psi(-x)\), the point $b_0=0$ and such a duality relation holds.
\par
\noindent
\begin{proposition}
If \(f=e^{-\psi}\) is an even log-concave function, then \( (L(f))^\circ=J(f^\circ)\). 
\end{proposition}
\begin{proof}
Let \( L(f)=e^{-\|T_0(x+b_0)\|_2+t_0}\) be the L\"owner function of \(f\). By Theorem \ref{lownerfcn}, \( (T_0, b_0, t_0)\) are the unique solution, modulo $O(n)$,  to the optimization problem 
  \[ n!\vol_n(B_2^n) \  \min_{b\in \R^n} \min\left\{\frac{e^t}{\det T}: T\in S_+, t\in \R, e^{-t}\1_{TB_2^n}(y)\le (f_b)^\circ(y) \right\},\]
  where \( f_b(x)=S_{-b}f \).
As \(f\) is even, \(b_0=0\). Hence the above minimum is obtained when \(b=0\), that is,
\begin{eqnarray*}
& &	\min_{b\in \R^n}\min\left\{\frac{e^t}{\det T}: T\in S_+, t\in \R, e^{-t}\1_{TB_2^n}(y)\le (f_b)^\circ(y)\right\}\\
&=& \min \left\{ \frac{e^t}{\det T}: T\in S_+, t\in \R, e^{-t}\1_{TB_2^n}(y)\le f^\circ \right\}\\
&=& \left( \max \left\lbrace  e^{-t}\det T: T\in S_+, t\in \R, e^{-t}\1_{TB_2^n}(y)\le f^\circ\right\rbrace \right) ^{-1}.
\end{eqnarray*}
In other words, \( (T_0, t_0)\) also solves 
\be 
\max \left\lbrace  e^{-t}\det T: T\in S_+, t\in \R, e^{-t}\1_{TB_2^n}(y)\le f^\circ\right\rbrace.    \label{evenjohnfcn3}
\ee
Now observe that \( f^\circ\) is an even function. In fact, since \( \psi(x)=\psi(-x)\), 
\begin{eqnarray*}
	\mathcal{L}(\psi)(-y) &=& \sup_{x\in \R^n} \langle x,-y \rangle -\psi(x)
	                      = \sup_{x\in \R^n} \langle -x,-y \rangle -\psi(-x)
	                      = \sup_{x\in \R^n} \langle x,y \rangle -\psi(x)\\
	                      &=& \mathcal{L}(\psi)(y).	                    
\end{eqnarray*}
Thus, \( f^\circ(-y)=e^{-\mathcal{L}(\psi)(-y)}=e^{-\mathcal{L}(\psi)(y)}=f^\circ(y) \).
By the eveness of \(f^\circ\), the maximum
\be 
\max_{b\in \R^n}\max \left\lbrace  e^{-t}\det T: T\in S_+, t\in \R, e^{-t}\1_{TB_2^n+b}(y)\le f^\circ\right\rbrace \label{evenjohnfcn4}
\ee
is achieved at the same solution to (\ref{evenjohnfcn3}). But the solution to (\ref{evenjohnfcn4}) gives the John ellipsoid function of \( f^\circ\). Therefore, 
\( J(f^\circ)=e^{-t_0}\1_{T_0B_2^n}\). It follows from a routine computation that 
\[ (L(f))^\circ=\left( e^{-\|T_0 x\|_2+t_0}\right)^\circ = e^{-t_0}\1_{T_0B_2^n}=J(f^\circ).\] 
\end{proof}
\vskip 3mm
\noindent
However, it is  not true in general that  \( L(f)^{b_0}=J(f^{b_0})\) or \( L(f^{b_0})=J(f)^{b_0}\).
We give a $1$-dimensional counter example. The higher dimensional counter example
is constructed accordingly.
\par 
\noindent
{\bf A counter example}. Let $f(x) = e^{-\psi(x)}$ be the log-concave function such that 
\begin{equation*} 
\psi(x) = 
\begin{cases}  4 x^2  & x \leq 0 \\
 x^2     & x > 0.  \end{cases}
 \end{equation*}
We compute that the L\"owner function of \(f\) is 
$$
L(f) = e^{- \frac{4}{\sqrt{5}}\left |x- \frac{3}{8 \sqrt{5}}\right| + \frac{1}{2}}
$$
and that the polar of $L(f)$ with respect to $ \frac{3}{8 \sqrt{5}}$ is 
$$
\left(L(f)\right)^{ \frac{3}{8 \sqrt{5}}} = e^{-\frac{1}{2} }\1_{[-\frac{4}{\sqrt{5}}, \frac{4}{\sqrt{5}}] + \frac{3}{8 \sqrt{5}}}.
$$
The polar of $f$ with respect to $ \frac{3}{8 \sqrt{5}}$ is
$$
\left(f\right)^{ \frac{3}{8 \sqrt{5}}} = e^{ \frac{3}{8\sqrt{5}}\left (x- \frac{3}{8 \sqrt{5}}\right) - \frac{1}{16} \left (x- \frac{3}{8 \sqrt{5}}\right)^2 } \1_{\left(-  \infty, \frac{3}{8\sqrt{5}}\right] } +
e^{ \frac{3}{8\sqrt{5}}\left (x- \frac{3}{8 \sqrt{5}}\right) - \frac{1}{4} \left (x- \frac{3}{8 \sqrt{5}}\right)^2 } \1_{\left(\frac{3}{8\sqrt{5}}, \infty\right) }.
$$
To find the John ellipsoid $J\left(\left(f\right)^{ \frac{3}{8 \sqrt{5}}}\right)$ of $\left(f\right)^{ \frac{3}{8 \sqrt{5}}}$ we  determine the super-level sets of $\left(f\right)^{ \frac{3}{8 \sqrt{5}}}$,
\begin{eqnarray*}
G_{\left(f\right)^{ \frac{3}{8 \sqrt{5}}}} (s) &=& \left\{x :  \left(f\right)^{ \frac{3}{8 \sqrt{5}}} \geq s\right\} \\
&=& \begin{cases} 
\left[\frac{3}{8 \sqrt{5}} + \frac{3- \left(9 - 80 \log s\right)^\frac{1}{2}}{\sqrt{5}},  \frac{3}{8 \sqrt{5}} + \frac{3+ \left(9 - 320 \log s\right)^\frac{1}{2}}{4 \sqrt{5}} \right], s \leq 1 \\
\left[\frac{3}{8 \sqrt{5}} + \frac{3- \left(9 - 320 \log s\right)^\frac{1}{2}}{4\sqrt{5}},  \frac{3}{8 \sqrt{5}} + \frac{3+ \left(9 - 320 \log s\right)^\frac{1}{2}}{4 \sqrt{5}} \right], s \geq 1 
\end{cases}\\
\end{eqnarray*}
and then maximize the function
\begin{eqnarray*}
h(s) = 
\begin{cases} 
\frac{s}{4 \sqrt{5}} \left( 4\left(9 - 80 \log s\right)^\frac{1}{2} + \left(9 - 320 \log s\right)^\frac{1}{2} -9 \right),  s \leq 1\\
 \frac{s}{2 \sqrt{5}} \left( 9 - 320 \log s\right)^\frac{1}{2}, s \geq 1.\\
\end{cases}\\
\end{eqnarray*}
If it were so that 
$$\left(L(f)\right)^{ \frac{3}{8 \sqrt{5}}} = e^{-\frac{1}{2} }\1_{[-\frac{4}{\sqrt{5}}, \frac{4}{\sqrt{5}}] + \frac{3}{8 \sqrt{5}}}= J\left(\left(f\right)^{ \frac{3}{8 \sqrt{5}}}\right),
$$
then the function $h$ would have its  maximum at $s=e^{-\frac{1}{2}} $ and  thus the derivative of $h$ at $s=e^{-\frac{1}{2}}$ should be $0$.
But $h^\prime \left(e^{-\frac{1}{2}}\right)\simeq -0.3538<0$.

\vskip 3mm
\noindent
{\bf Acknowledgement}
\par
\noindent
This material is based upon work supported by the National Science Foundation under Grant No. DMS-1440140 while the  authors were in residence at the Mathematical Sciences Research Institute in Berkeley, California, during the Fall 2017 semester. 
\par
\noindent
The authors want to thank F. Mussnig and the referee for the careful reading of the manuscript, the suggested improvements and corrections.

{}

\vskip 4mm
\noindent
Ben Li\\
{\small School of Mathematical Sciences}\\
{\small Tel Aviv University}\\
{\small Tel Aviv 69978,
	Israel}\\
{\small \tt liben@mail.tau.ac.il}\\ \\
\vskip 3mm
\noindent
Carsten Sch\"utt\\
{\small Mathematisches  Institut}\\
{\small Universit\"at Kiel}\\
{\small Germany }\\
{\small \tt schuett@math.uni-kiel.de}\\ \\
\vskip 3mm
\noindent
Elisabeth M. Werner\\
{\small Department of Mathematics \ \ \ \ \ \ \ \ \ \ \ \ \ \ \ \ \ \ \ Universit\'{e} de Lille 1}\\
{\small Case Western Reserve University \ \ \ \ \ \ \ \ \ \ \ \ \ UFR de Math\'{e}matique }\\
{\small Cleveland, Ohio 44106, U. S. A. \ \ \ \ \ \ \ \ \ \ \ \ \ \ \ 59655 Villeneuve d'Ascq, France}\\
{\small \tt elisabeth.werner@case.edu}\\ \\


\begin{thebibliography}{10}
	





	
	
	\bibitem{Alonso-Gutiérrez2017}
	D. Alonso-Guti{\'e}rrez,  B. G. Merino, 
	C. H. Jim{\'e}nez, 
	and R. Villa. 
	{\it John's Ellipsoid and the Integral Ratio of a Log-Concave Function}.
	Journal of Geometric Analysis,  28: 1182--1201, 2018. 
	
	
	\bibitem {AGM2014}
	S. Artstein-Avidan, A. Giannopoulos, and V. Milman. 
	{\it Asymptotic geometric analysis}. Part I, volume 202 of Mathematical Surveys and Monographs. American Mathematical Society, Providence, RI, 2015.
	
	\bibitem {AKM2004}
	S. Artstein-Avidan, B. Klartag,  and V. Milman. 
	{\it The {S}antal\'o point of a function, and a functional form of
		the {S}antal\'o inequality}.
	 Mathematika,   51: 33--48, 2004.
	
	\bibitem {AKSW2012}
	S. Artstein-Avidan, B. Klartag, C. Sch\"utt, and E. M. Werner. 
	{\it Functional affine- isoperimetry and an inverse logarithmic Sobolev inequality}. 
	J.\ Funct.\ Anal.\  262: 4181--4204, 2012.
	
	
	
	\bibitem{ASW10}
       G. Aubrun, G, S. J. Szarek and E. M. Werner.
        {\it Nonadditivity of R\'enyi entropy and Dvoretzky's theorem}, 
    J. Math. Phys., 51, 2010.
        
     \bibitem{ASW11}
        G. Aubrun, G, S. J. Szarek and E. M. Werner.
        {\it Hastings's Additivity Counterexample via Dvoretzky's Theorem},
       Commun. Math. Phys. 305: 85--97, 2011.
        
	\bibitem {Ball1988}
	K. Ball.
	{\it  Logarithmically concave functions and sections of convex sets in $\R^n$ }.  Studia Math., 88(1): 69--84, 1988. 
	
	\bibitem {Ball1989}
	K. Ball.
	{\it  Volumes of sections of cubes and related problems}.  In Geometric aspects of functional analysis (1987--88), volume 1376 of Lecture Notes in Math., pages 251--260. Springer, Berlin, 1989. 
	
	\bibitem {Ball1991}
	K. Ball.
	{\it  Volume ratios and a reverse isoperimetric inequality}. J. London Math. Soc. (2), 44(2): 351--359, 1991. 
	
	\bibitem {Fbarthe1998}
	F. Barthe.
	{\it On a reverse form of the Brascamp-Lieb inequality}. Invent. Math., 134(2): 335--361, 1998. 
	
	\bibitem{Bourgain1991}
J. Bourgain.
{\it On the distribution of polynomials on high-dimensional convex sets.}
Israel Seminar on GAFA, Lindenstrauss, Milman (Eds.),  Springer Lecture Notes  1469: 127--137, 1991. 

	
	
	\bibitem{BourgainMilman87}
   J. Bourgain, and V. D. Milman. 
   {\it New volume ratio properties for convex symmetric bodies in $\R^n$},
   Invent. Math., 
   88: 319--340, 1987.

	
	
	
	\bibitem {BGVV2014Book}
	S. Brazitikos, A. Giannopoulos, P. Valettas, and B-H. Vritsiou. 
	{\it Geometry of isotropic convex bodies}. Volume 196 of Mathematical Surveys and Monographs. American Mathematical Society, Providence, RI, 2014. 
	
	\bibitem {CFGLSW}
	U. Caglar, M. Fradelizi, O. Gu\'edon, J. Lehec, C. Sch\"utt, and E. M. Werner.
	{\it Func- tional versions of Lp-affine surface area and entropy inequalities }. Int. Math. Res. Not. IMRN, (4): 1223--1250, 2016. 
	
	
	\bibitem{ColesantiLudwigMussnig2017}
	A. Colesanti, M. Ludwig, and F. Mussnig.
	{\it Minkowski valuations on convex functions}, Calc. Var. Partial Differential Equations
    56: 56--162, 2017.
	
	\bibitem{ColesantiLudwigMussnig}
	A. Colesanti, M. Ludwig, and F. Mussnig. 
	{\it Valuations on convex functions}, Int. Math. Res. Not. IMRN, in press.
	
	
	\bibitem {FM2007}
	M. Fradelizi and M. Meyer.
	{\it Some functional forms of Blaschke-Santal\'o inequality}. Math. Z., 256(2): 379--395, 2007. 
	
	\bibitem {FM2008}
	M. Fradelizi and M. Meyer.
	{\it Some functional inverse Santal\'o  inequalities}. Adv. Math., 218(5): 1430--1452, 2008. 
	
	\bibitem {GardnerBulletin}
	R. J. Gardner.
	{\it The Brunn-Minkowski inequality }. Bull. Amer. Math. Soc. (N.S.), 39(3): 355--405, 2002. 
	
	\bibitem {GardnerBook2006}
	R. J. Gardner.
	{\it Geometric Tomography }. Number v. 13 in Encyclopedia of Mathematics an. Cambridge University Press, 2006.
	
	\bibitem{GPSTW} O. Giladi, J. Prochno, C. Sch\"utt,  N. Tomczak-Jaegermann and E.M. Werner. 
	{\it On the geometry of projective tensor products},  Journal of Functional Analysis
273: 471--495, 2017.  
	
	
	
	
	
	\bibitem {Gluskin1981}
	E. D. Gluskin. 
	{\it The diameter of the Minkowski compactum is roughly equal to n}. Funktsional. Anal. i Prilozhen., 15(1): 72--73, 1981. 
	
	
	
	\bibitem{Gruenbaum1963}
 B. Gr\"unbaum.
{\it Measures of symmetry for convex sets},
Proc. Sympos. Pure Math. 7: 233--270, 1963.

\bibitem{Haberl:2012}
{ C. Haberl}. {\em Minkowski valuations intertwining the special linear group},  
J. Eur. Math. Soc. (JEMS)  14: 565--1597, 2012.
	
	
	\bibitem {henkJLellip}
	M. Henk.
	{\it L\"owner-John ellipsoids }. Doc. Math., (Extra vol.: Optimization stories): 95--106, 2012. 
	
	
	\bibitem {john1948}
	F. John.
	{\it Extremum problems with inequalities as subsidiary conditions }. In Studies and Essays Presented to R. Courant on his 60th Birthday, January 8, 1948, pages 187--204. Interscience Publishers, Inc., New York, N. Y., 1948. 
	
	\bibitem {KlartagMilman2005}
	B. Klartag and V. D. Milman.
	{\it Geometry of log-concave functions and measures}. Geom. Dedicata, 112: 169--182, 2005. 
	
	
	
	\bibitem{Klartag06}
 B. Klartag. 
{\it On convex perturbations with a bounded isotropic constant}.
Geom. and Funct. Anal. (GAFA), Vol. 16, Issue 6: 1274--1290, 2006. 

	\bibitem{KlartagWerner}
	B. Klartag and E. M. Werner.
	 {\it Some open problems in Asymptotic Geometric Analysis}  June/July 2018 Notices of the AMS, (2018).
	
	
	


\bibitem{Ludwig:2010c}
M. Ludwig. 
{\it Minkowski areas and valuations},
J. Differential Geom., 
 86: 133--161, 2010.





\bibitem{Ludwig:2010}
M.~Ludwig and M.~Reitzner. 
{\it A classification of ${\rm SL}(n)$ invariant valuations},
Ann.\ of Math.\ (2)  172: 1219--1267, 2010.

	
	
	
	
	
	\bibitem {Leh2009}
	J. Lehec.
	{\it A direct proof of the functional Santal\'o inequality}. C. R. Math. Acad. Sci. Paris, 347(1-2): 55--58, 2009. 
	
	\bibitem {MSW2015b}
	M. Meyer, C. Sch\"utt, and E. M. Werner.
	{\it Dual affine invariant points}. Indiana Univ. Math. J., 64(3): 735--768, 2015. 
	
	
	
	\bibitem{MordhorstWerner1}
O. Mordhorst, E. M. Werner, {\it Duality of Floating and Illumination Bodies}. To appear in Indiana Univ. Math. J.
	
	
	\bibitem{MordhorstWerner2}
O. Mordhorst, E. M. Werner, {\it Floating and Illumination Bodies for Polytopes: Duality Results}. To appear in Journal Discrete Analysis.

	
	
	
	
	
	\bibitem{Mussnig2017}
	F. Mussnig. 
	{\it Valuations on log-concave functions}, preprint, arXiv:1707.06428, 2017.

	
	
	
	\bibitem {PisierBook1989}
	G. Pisier.
	{\it The volume of convex bodies and Banach space geometry}. Volume 94 of Cambridge Tracts in Mathematics. Cambridge University Press, Cambridge, 1989. 
	
	
\bibitem{RobertsVarbergBook1973}
	A. W. Roberts and D. E. Varberg. 
	{\it Convex functions}. Academic Press [A subsidiary of Harcourt Brace Jovanovich, Publishers], New York-London, 1973. Pure and Applied Mathematics, Vol. 57.	
	
		
	\bibitem {RockafellarBook1970}
	R. T. Rockafellar. 
	{\it Convex analysis}. Princeton Mathematical Series, No. 28. Princeton University Press, Princeton, N.J., 1970. 
	
	\bibitem {RockafellarWetsBook1997}
	R. T. Rockafellar, R. J-B. Wets. 
	{\it Variational  analysis}. Springer-Verlag in the series Grundlehren der Mathematischen Wissenschaft, 1997.
	
	
	\bibitem {SchneiderBook}
	R. Schneider. 
	{\it Convex Bodies: The Brunn-Minkowski Theory}. Encyclopedia of Mathematics and its Applications. Cambridge University Press, 2014. 
	
\bibitem{Schuett1982}
C.~Sch\"utt.
{\it On the volume of unit balls in Banach spaces}, Compositio Math.  47: 393--407,  1982.

\bibitem{Schuster2010}
F.~Schuster. {\it Crofton measures and Minkowski valuations}, Duke Math. J.
154: 1--30, 2010.


\bibitem{SchusterWannerer}
F.~Schuster and M.~Weberndorfer. {\it Minkowski Valuations and Generalized Valuations}
Journal of the European Mathematical Society, in press.


	\bibitem{Szarek78}
   S. J. Szarek.
   {\it On Kashin's almost Euclidean orthogonal decomposition of
   $l^{1}_{n}$},
  Bulletin de l'Acad\'emie polonaise des Sciences   26: 691--694, 1978.
   
\bibitem{SWZ11}
        S. J. Szarek, E. M. Werner and K. Zyczkowski.
        {\it How often is a random quantum state $k$--entangled?}
 J. Phys. A,  40: 44, 2011.
  
\bibitem{SzarekTomczakJ80}
   S. J. Szarek and 
   N. Tomczak-Jaegermann.
   {\it On nearly Euclidean decomposition for some classes of Banach
   spaces},
  Compositio Math.
   40: 367--385,
 1980.
 	

	
	\bibitem{Werner1994}
	E. M. Werner. 
	{\it Illumination bodies and affine surface area}. Studia Math., 110(3): 257--269, 1994. 
	
	
	
	
\end{thebibliography}
 \end{document}